\pdfoutput=1

\documentclass[11pt]{amsart}
\usepackage{fullpage }   

\usepackage[centertags]{amsmath}
\usepackage{amsfonts}
\usepackage{amssymb}
\usepackage{amsthm}
\usepackage{ subfigure}
\usepackage{graphicx}
\usepackage{color}
\usepackage{enumerate}

\newtheorem{theorem}{Theorem}[section]
\newtheorem{proposition}[theorem]{Proposition}
\newtheorem{lemma}[theorem]{Lemma}

\newtheorem{corollary}[theorem]{Corollary}
\newtheorem{propdef}[theorem]{Proposition/Definition}

\theoremstyle{definition}
\newtheorem{definition}[theorem]{Definition}

\newtheorem{example}[theorem]{Example}

\theoremstyle{remark}

\numberwithin{equation}{section}

\begin{document}






\newcommand{\BC}{\mathrm{BC}}

\newcommand{\ba}{{\bf a}}

\newcommand{\bal}{{\boldsymbol\alpha}}
\newcommand{\bbe}{{\boldsymbol\beta}}
\newcommand{\bga}{{\boldsymbol\gamma}}

\newcommand{\al}{ \alpha}
\newcommand{\bet}{ \beta}
\newcommand{\ga}{ \gamma}
\newcommand{\la}{ \lambda}

\newcommand{\g}{\xi}
\newcommand{\h}{\zeta}

\newcommand{\vg}{\boldsymbol\xi}
\newcommand{\vh}{\boldsymbol\zeta}

\newcommand{\fG}{\mathfrak{G}}
\newcommand{\fS}{\mathfrak{S}}
\newcommand{\calG}{\mathcal{G}}

\newcommand{\vs}{\vec{s}\,{}^{\prime} }
\newcommand{\vG}{\vec{s}\,{}^{\Gamma} }

\newcommand{\pd}{\vec{G}_m^P}


\title[ A Penrose polynomial for embedded graphs ]{A Penrose polynomial for embedded graphs }



\author[J.~Ellis-Monaghan]{Joanna A. Ellis-Monaghan}
\address{Department of Mathematics, Saint Michael's College, 1 Winooski Park, Colchester, VT 05439, USA.  }
\email{jellis-monaghan@smcvt.edu}
\thanks{The work of the first author was supported by the National Science Foundation (NSF) under grant number DMS-1001408, by the Vermont Space Grant Consortium through the National Aeronautics and Space Administration (NASA), and by the Vermont Genetics Network through Grant Number P20 RR16462 from the INBRE Program of the National Center for Research Resources (NCRR), a component of the National Institutes of Health (NIH).  This paper's contents are solely the responsibility of the authors and do not necessarily represent the official views of the NSF, NASA, NCRR, or NIH}

\author[I.~Moffatt]{ Iain Moffatt}
\address{Department of Mathematics and Statistics,  University of South Alabama, Mobile, AL 36688, USA}
\email{imoffatt@southalabama.edu}



\date{\today}

\begin{abstract}  
We extend the Penrose polynomial, originally defined  for plane graphs, to graphs embedded in  arbitrary surfaces. Considering this Penrose polynomial of embedded graphs leads to new identities and relations for the Penrose polynomial which can not be realized within the class of plane graphs.  
In particular, by exploiting connections with the transition polynomial and the ribbon group action, we find a deletion-contraction-type relation for the Penrose polynomial. We relate the Penrose polynomial of an orientable checkerboard colourable graph to the circuit partition polynomial of its medial graph and use this to find new combinatorial interpretations of the Penrose polynomial.  We also show that the Penrose polynomial of a plane graph $G$ can be expressed as a sum of chromatic polynomials of  twisted duals of $G$. This allows us to obtain a new reformulation of the Four Colour Theorem.
\end{abstract}

%


\maketitle

\section{Introduction}\label{s.intro}

The Penrose polynomial, $P(G,\lambda)$, of a plane graph first appeared implicitly in \cite{Pen71}, where it arose out of Penrose's work on diagrammatic tensors. The Penrose polynomial has a number of remarkable graph theoretical properties, 
particularly with respect to graph colouring.  
For example, it is well-known that The Four Colour Theorem is equivalent
to showing that every plane, cubic, connected graph can be properly
edge-coloured with three colours.   The Penrose polynomial encodes exactly this information
(see Penrose~\cite{Pen71}): if $G$ is a plane, cubic, connected graph, then

\begin{equation}\label{e.ptcol}
\text{the number of edge 3-colourings of } G = P\left( {G;3} \right) = \left(  - 1/4
\right)^{\frac{{v(G)}}{2}} P\left( {G; - 2} \right) .
\end{equation}

In addition to this relation with the Four Colour Theorem, the Penrose polynomial has a host of other significant graph theoretical properties. We refer the reader to the excellent  expositions on the Penrose polynomial given by Aigner in~\cite{Ai97} and~\cite{Aig00}, and also to the papers \cite{AM00, E-MS01, Sar01, Sze02} for further explorations of its properties.

In this paper we extend the Penrose polynomial, previously defined only for plane graphs, to graphs embedded in arbitrary surfaces.
This leads to a number of new properties of the Penrose polynomial of both plane graphs and embedded graphs.  Our primary tools for investigating the Penrose polynomial are   
 twisted duality and the ribbon group action which were introduced by the authors in \cite{E-MMe}, where they were used to solve isomorphism problems involving medial graphs and to expose the connections among embedded graphs and their medial graphs.  This ribbon group action depends on two operations:  $\tau(e)$ which gives a half-twist to an edge $e$, and $\delta(e)$ which forms the partial dual with respect to an edge $e$. The connection between the ribbon group action and the Penrose polynomial arises through the transition polynomial.     It has been shown that the Penrose polynomial of a plane graph can be expressed in terms of the transition polynomial (see e.g. \cite{Ja90}). We will show that, analogously, the Penrose polynomial of an embedded graph can be written in terms of  the topological transition polynomial defined in \cite{E-MMe}. Also, in \cite{E-MMe}, it was shown that the ribbon group acts on embedded graphs and on the topological transition polynomial in a compatible way. It is this connection between operations on embedded graphs and graph polynomials that we exploit to determine properties of the Penrose polynomial.

Many of the properties we describe here can not be realized within the class of plane graphs in that a relation for a graph $G$ may involve graphs embedded in different surfaces than $G$. For example, we will show that, for embedded graphs, the Penrose polynomial satisfies the deletion-contraction relation 
\[P(G^{\delta\tau(e)};\la)= P(G/e; \la)-P(G-e; \la)\], 
where $G^{\delta\tau(e)}$ is a twisted dual of $G$, and is generally non-plane.  We also show that the Penrose polynomial of a plane  graph $G$ can be expressed as a sum of the  chromatic polynomials of the duals of its partial Petrials, where a partial Petrial results from giving a half-twist to some subset of the edges of a graph:
\begin{equation}\label{penchrom1}
P(G;\lambda) = \sum_{A\subseteq E(G)}  \chi ((   G^{\tau(A)}   )^*   ;\lambda)  .
\end{equation} 
This identity completes a result of Aigner from \cite{Ai97}, which states that 
 \[P(G;k)\leq \chi(G^*;k),\]
for all $k\in \mathbb{N}$. The expression $\chi(G^*;k)$ in Aigner's inequality is a single summand in our expression (\ref{penchrom1}) for the Penrose polynomial $P(G;k)$.

This relation between the chromatic and Penrose polynomials allows us to give a new reformulation of the Four Colour Theorem in terms of the chromatic polynomial.

Further to these results, which illustrate the advantage of considering the Penrose polynomial for graphs embedded in an arbitrary surface, rather than just the plane, we also discuss which of the known properties of the Penrose polynomial for plane graphs extend to non-plane graphs.  In particular, we relate the Penrose polynomial of an orientable checkerboard colourable graph to the circuit partition polynomial of its medial graph and use this to find new combinatorial interpretations of the Penrose polynomial.  

This paper is structured as follows. We begin with some preliminary definitions for embedded graphs. Then, in Section~\ref{s.penrose}, we define the Penrose polynomial  of an embedded graph, discuss some differences between the plane and  non-plane Penrose polynomial, and describe the relation between the topological transition polynomial and the Penrose polynomial, which is of fundamental importance here.
 In Section~\ref{s.td}, we give a reasonably comprehensive overview of  the ribbon group action and twisted duality, our main tools in this paper, and describe how the ribbon group interacts with the transition polynomial. In Section~\ref{s.tdpenrose} we apply our results on twisted duality to to obtain identities for the Penrose polynomial, including several deletion-contraction relations. Section~\ref{s.cpp} contains the relation to the circuit partition polynomials and new combinatorial interpretations for evaluations of the Penrose polynomial.  In Section~\ref{s.colour}, by examining connections between the Penrose polynomial and graph colourings we express the Penrose polynomial as a sum of the chromatic polynomials of twisted duals. We conclude with a new reformulation of the Four Colour Theorem.  

We would like tho thank Gabriel Cunningham for helpful comments.

\section{Preliminary definitions}\label{s.def}

\subsection{Embedded graphs and medial graphs}\label{ss.defs}

A {\em cellularly embedded graph} is a graph  $G$ embedded in a surface $\Sigma$ such that every face is a 2-cell. If $G\subset \Sigma$, the homeomorphism class of the surface $\Sigma$ generates an equivalence class of embedded graphs, and we say that embedded graphs are  {\em equal} if they are in the same equivalence class.   We assume familiarity with embedded graphs, referring the reader to \cite{GT87} for further details.

 We use standard notation: $V(G)$, $E(G)$, and $F(G)$ are the sets of vertices, edges, and faces, respectively, of a cellularly embedded graph $G$, while $v(G)$, $e(G)$, and $f(G)$, respectively, are the numbers of such.
We will say that a loop ({\em i.e.}  an edge that is incident to exactly one vertex) in a cellularly embedded graph is {\em non-twisted} if a neighbourhood of it is homeomorphic to an annulus, and we say that the loop is {\em twisted} if a neighbourhood of it is homeomorphic to a M\"obius band.

\subsection{Representing cellularly embedded graphs}\label{ss.rg}

There are several ways to represent cellularly embedded graphs, and it is often more convenient and natural to work in the language of one or the other of these representations. Thus, we briefly describe ribbon graphs, arrow presentations and their equivalence to each other and to signed rotation systems and cellularly embedded graphs.   Further details may be found in \cite{Ch1,GT87}.

\begin{definition}
A {\em ribbon graph} $G =\left(  V(G),E(G)  \right)$ is a (possibly non-orientable) surface with boundary, represented as the union of two  sets of topological discs, a set $V (G)$ of {\em vertices}, and a set $E (G)$  of {\em edges} such that: 
\begin{enumerate} 
\item the vertices and edges intersect in disjoint line segments;
\item each such line segment lies on the boundary of precisely one
vertex and precisely one edge;
\item every edge contains exactly two such line segments.
\end{enumerate}
\end{definition}

Two ribbon graphs are said to be {\em equivalent} or {\em equal} if there is a homeomorphism between them that preserves the vertex-edge structure.
 The {\em genus} of a ribbon graph  is its genus as a punctured surface.  A ribbon graph is said to be  {\em plane} if it is of genus zero.  In the context of ribbon graphs, $f(G)$ is the number of boundary components of the surface comprising the ribbon graph $G$.

It is well known that ribbon graphs and cellularly embedded graphs are equivalent:
if $G$ is a cellularly embedded graph, a ribbon graph representation results from taking a small neighbourhood  of the embedded graph $G$. Neighbourhoods  of vertices of the graph $G$ form the vertices of a ribbon graph, and neighbourhoods of the edges of the embedded graph form the edges of the ribbon graph.  On the other hand, if $G$ is a ribbon graph, we simply cap off the punctures to obtain a ribbon graph cellularly embedded in the surface. Form an embedded graph by placing one vertex in each vertex disc and connect them with edges that follow the edge discs. See Figure~\ref{f.desc}, which shows a graph embedded in the projective plane.

\begin{figure}
\[ \raisebox{-4mm}{\includegraphics[height=2.4cm]{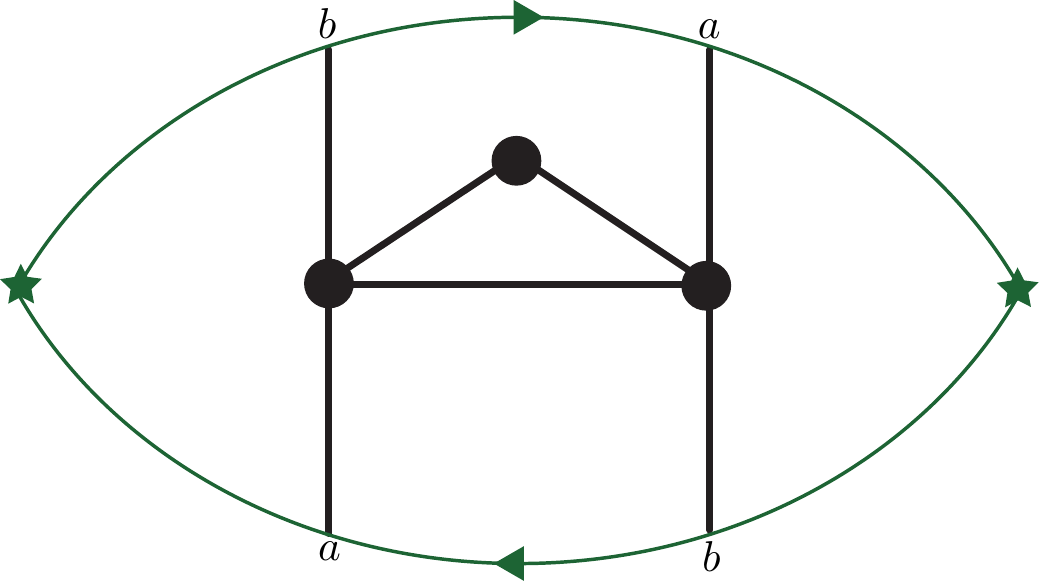}}  \quad \raisebox{7mm}{=}  \quad  \raisebox{-7mm}{\includegraphics[height=3cm]{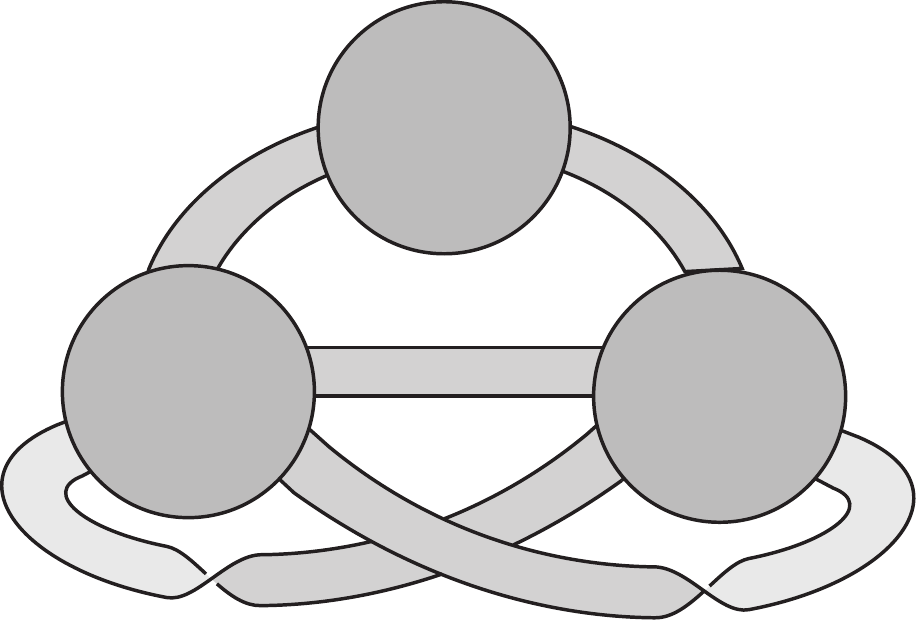}}  \quad \raisebox{7mm}{=} \quad \includegraphics[height=1.6cm]{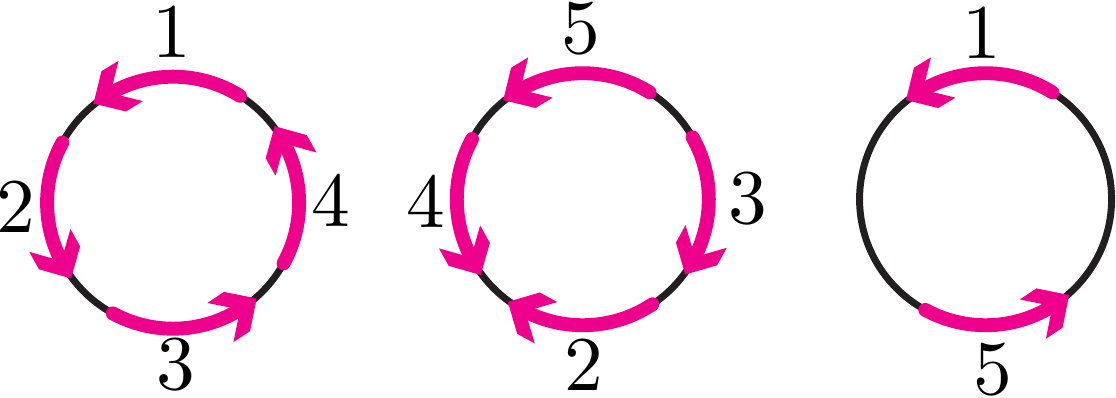}  \]

\caption{Three descriptions of the same embedded graph.}
\label{f.desc}
\end{figure}

\medskip

We will also find the following, more combinatorial, description of a cellularly embedded graph useful.
\begin{definition}[Chmutov \cite{Ch1}]
An {\em arrow presentation} consists of  a set of circles,  each with a collection of disjoint,  labelled arrows, called {\em marking arrows}, lying on them. Each label appears on precisely two arrows.
\end{definition}

 Two arrow presentations are {\em equivalent}  if one can be obtained from the other by reversing the direction of all of the marking arrows which belong to some subset of labels, or by relabelling the pairs of arrows.

A ribbon graph can be obtained from an arrow presentation by viewing each circle as the boundary of a disc that becomes a vertex of the ribbon graph.  Edges are then added to the vertex discs by taking a disc for each label of the marking arrows.  Orient the edge discs arbitrarily and choose two non-intersecting arcs on the boundary of each of the edge discs. Orient these arcs according to the orientation of the edge disc. 
Finally, identify these two arcs with two marking arrows, both with the same label, aligning the direction of each arc consistently with the orientation of the marking arrow. This process is illustrated pictorially thus:
\[ \includegraphics[height=10mm]{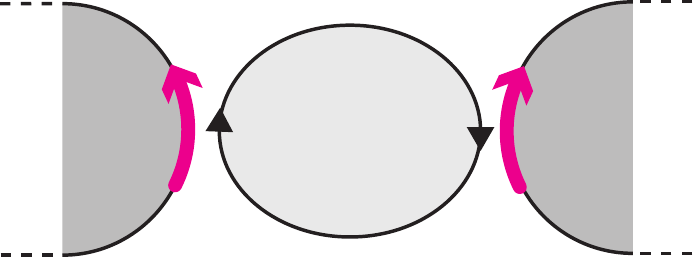}  \hspace{4mm}
\raisebox{4mm}{\includegraphics[width=11mm]{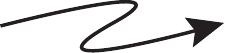}}\hspace{4mm}
\includegraphics[height=10mm]{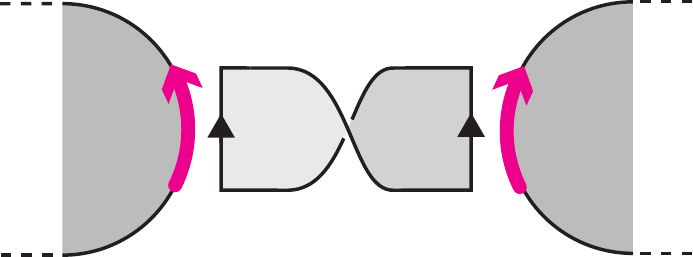} \hspace{4mm}
\raisebox{4mm}{\includegraphics[width=11mm]{arrow}}\hspace{4mm}
 \includegraphics[height=10mm]{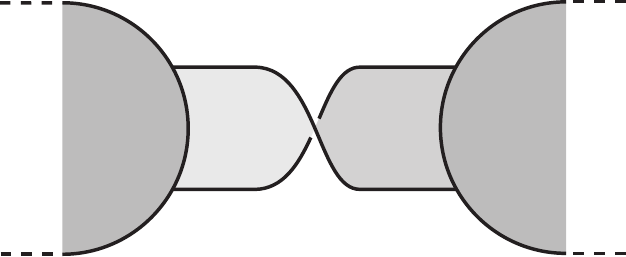} \; . \]

Conversely, every ribbon graph  gives rise to an arrow presentation.  To describe a ribbon graph $G$ as an arrow presentation, start by arbitrarily orienting and labelling  each edge disc in $E(G)$. This induces an orientation on the boundary of each edge in $E(G)$.  Now, on the arc where an edge disc intersects a vertex disc, place a marked arrow on the vertex disc, labelling the arrow with the label of the edge it meets and directing it consistently with the orientation of the edge disc boundary. The boundaries of the vertex set marked with these labelled arrows give the arrow marked circles of an arrow presentation.  See Figure~\ref{f.desc} for an example, and \cite{Ch1} for further details.

Arrow presentations are essentially signed rotation systems (see \cite{MT01} for example) with the labels about the circles giving the rotations and the agreement or disagreement of the directions of the marking arrows corresponding to the plus or minus signs on the edges.  

Since cellularly embedded graphs, ribbon graphs and arrow presentations are all equivalent, we can, and will, move freely between these representations, choosing whichever is most convenient at the time for our purposes.  We will use the term `embedded graph' loosely to mean any one of these representations of a cellularly embedded graph.

\subsection{Medial graphs}\label{ss.medial}

The Penrose polynomial of a plane graph can be defined via its medial graph, and we similarly will define the Penrose polynomial of an embedded graph via its medial graph.   Accordingly, medial graphs play a central role in this paper.  If  a graph $G$ is cellularly embedded, we construct its {\em medial graph}, $G_m$,  by placing a vertex of degree $4$ on each edge of $G$, and then drawing the edges of the medial graph by following the face boundaries of $G$. Consistent with this definition is that the medial graph of an isolated vertex is an isolated face, and we adopt this convention. Note that if $G$ is cellularly embedded in a surface $\Sigma$, then $G_m$ is also cellularly embedded in $\Sigma$.  

A {\em checkerboard colouring} (or face $2$-colouring)  of a $4$-regular embedded  graph $F$ is an assignment of the colour black or white to each of its faces such that adjacent faces have different colours. While not all $4$-regular embedded  graphs admit checkerboard colourings, all medial graphs do. In fact, we can associate a canonical checkerboard colouring  to a medial graph.
In the construction of $G_m$, the vertices of $G$ appear in some of the faces of $G_m$. Thus we can associate a face of $G_m$ with each vertex of $G$. We can then construct  a checkerboard colouring (or face $2$-colouring) of $G_m$ by colouring a face black if it is associated with a vertex of $G$, and colouring it white otherwise. This checkerboard colouring is called the {\em canonical checkerboard colouring} of $G_m$. 

 An example of a medial graph and its canonical checkerboard colouring  is given in Figure~\ref{fig.medialgraph}. 

\begin{figure}
\begin{tabular}{p{4cm}cp{4cm}cp{4cm}}
\includegraphics[height=25mm]{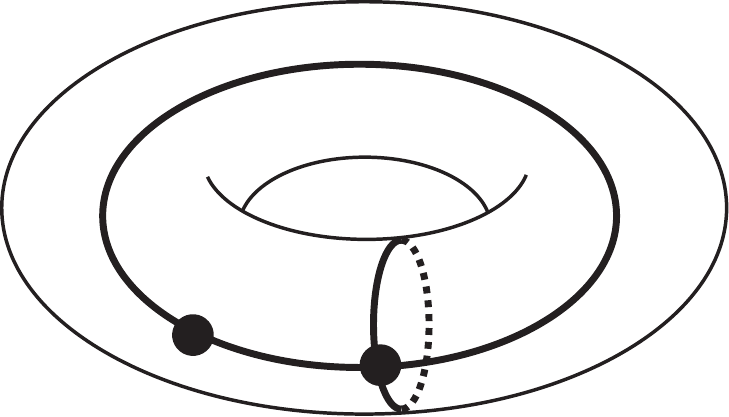} & \hspace{10mm} & \includegraphics[height=25mm]{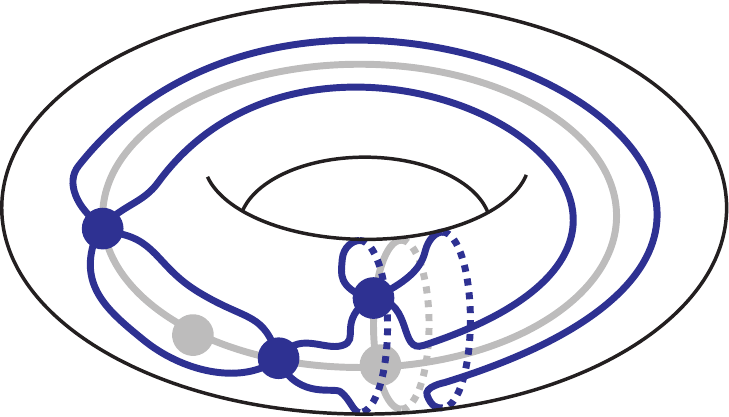}  & \hspace{10mm} &  \includegraphics[height=25mm]{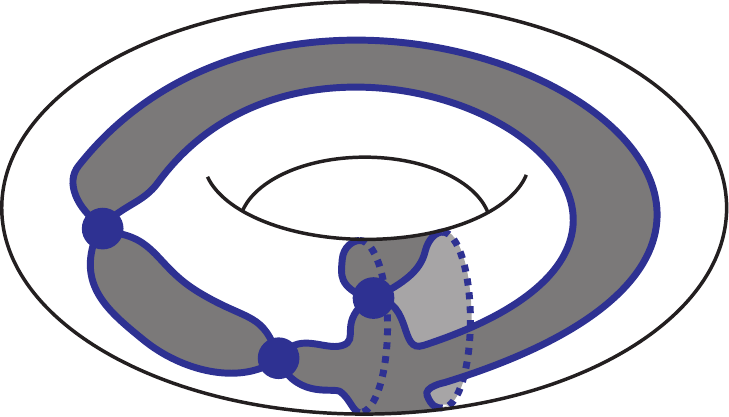}  \\ && &&\\
A cellularly embedded graph $G$. && The  medial graph $G_m$. &  & The canonical checkerboard colouring of $G_m$.
\end{tabular}
\caption{An example of a medial graph.}
\label{fig.medialgraph}
\end{figure}

\subsection{Weight systems and graph states}\label{ss.ws}

A \emph{vertex state} at a vertex $v$  of an abstract $4$-regular graph $F$ is a partition, into pairs, of the edges incident with $v$. 
Thus, if $F$ is an cellularly embedded graph, a vertex state is simply a choice of one of the following configurations in a neighbourhood of the vertex $v$:

\[\includegraphics[width=1.2cm]{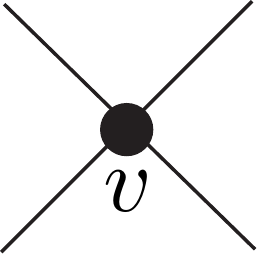}\quad \raisebox{5mm}{$\longrightarrow$} \quad \quad  
 \includegraphics[width=1.2cm]{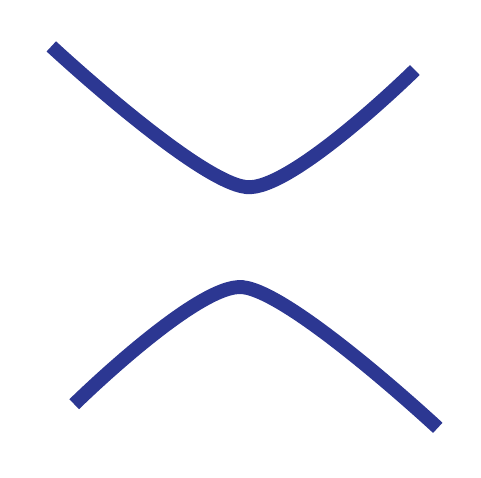}\quad \raisebox{5mm}{,} \quad \quad  \includegraphics[width=1.2cm]{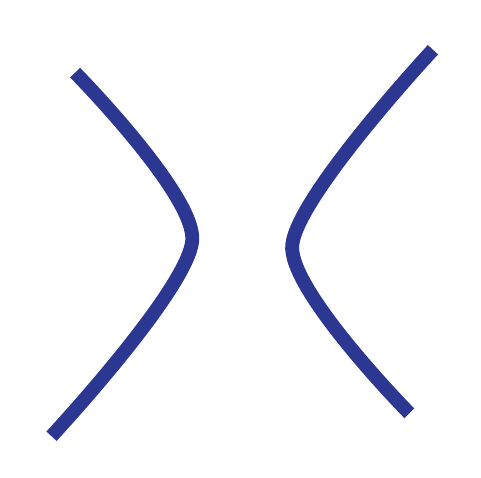}\quad \quad\raisebox{5mm}{ or } \quad \quad \includegraphics[width=1.2cm]{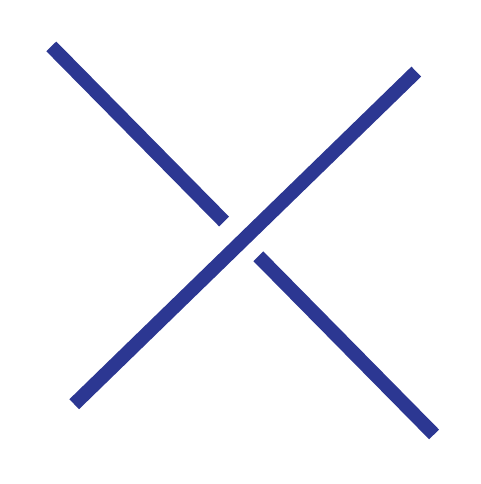}\raisebox{5mm}{.}\]
A choice of one of the configurations on the right replace a small neighbourhood of the vertex $v$.

 If $G$ is an embedded graph and  $G_m$ its canonically checkerboard coloured medial graph, then we may use the checkerboard colouring to distinguish among the vertex states, naming them a {\em white split}, a {\em black split} or a {\em crossing}, as follows:
 \begin{equation}\label{e.splits}
 \begin{array}{ccccccc}
 \includegraphics[height=12mm]{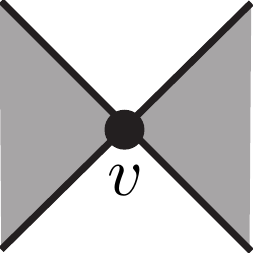} & \quad \raisebox{5mm}{$\longrightarrow$} \quad  & \includegraphics[height=12mm]{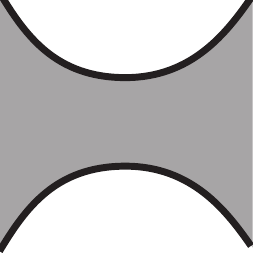} &\quad \raisebox{5mm}{or} \quad  & \includegraphics[height=12mm]{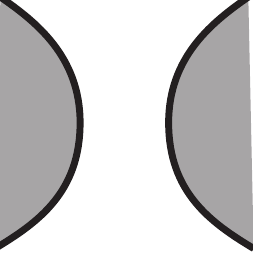} &\quad \raisebox{5mm}{or} \quad  &\includegraphics[height=12mm]{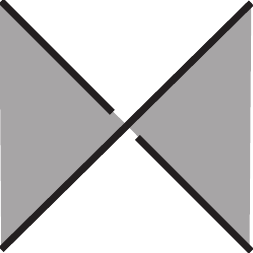} \\
\text{in }G_m && \text{white split} && \text{black split} && \text{crossing.}
 \end{array}
 \end{equation}

A {\em graph state} $s$ of  any 4-regular graph $F$ is a choice of vertex state at each of its vertices. Each graph state corresponds to a specific family of edge-disjoint cycles in $F$, and this family is independent of embedding (although different embeddings of $F$ will generally use different vertex states to generate the same family of disjoint cycles). We call these cycles the {\em components of the state}, denoting the number of them by $c(s)$.
We will denote the set of states of a $4$-regular graph $F$ by $St(F)$. If $G_m$ is the canonically checkerboard coloured medial graph of $G$, then $\mathcal{P}(G_m)$ denotes the set of states with no
  black splits, and we will call such states {\em Penrose states}.

\begin{example} Three of the possible nine graph states of the graph $G=\raisebox{-4mm}{ \includegraphics[height=1.2cm]{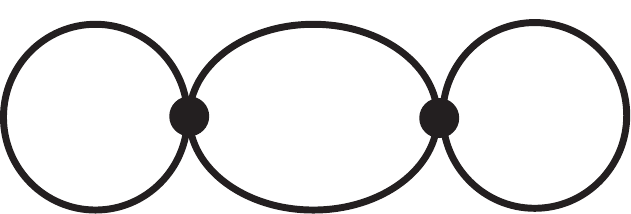}}$ are 
\raisebox{-4mm}{\includegraphics[height=1.2cm]{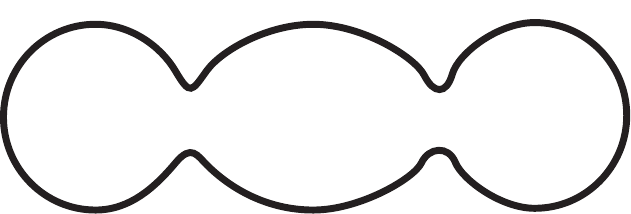}}\;, 
\raisebox{-4mm}{\includegraphics[height=1.2cm]{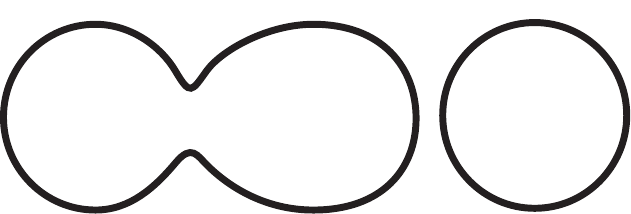}}\;,
and
\raisebox{-4mm}{\includegraphics[height=1.2cm]{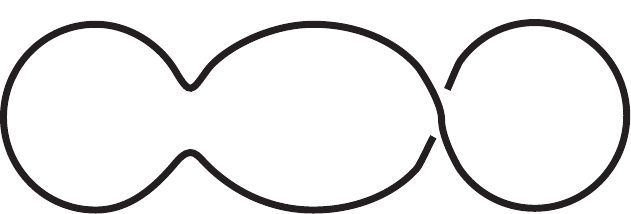}}\;.

\end{example}

A \emph{weight system}, $W( F )$, of any $4$-regular graph $F$ (embedded or not) is an assignment of a weight in a unitary ring $\mathcal{R}$ to every vertex state of $F$.     (We simply write $W$ for $W(F)$ when the graph is clear from context.)  If $s$ is a state of $F$, then the \emph{state weight} of $s$ is $\omega(s) := \prod_{v \in V(F)} {\omega(v,s)}$, where $\omega(v,s)$ is the vertex state weight of the vertex state at $v$ in the graph state $s$.  See \cite{E-MS02} for further details.

If $G_m$ is an embedded  medial graph, then we can construct a weight system by associating a weight to white splits, black splits and crossings as in the definition below. We will use the resulting medial weight system to define the Penrose polynomial of an embedded graph in Section \ref{s.penrose}.

\begin{definition}\label{d.mws}
 Let $G$ be an embedded graph with embedded medial graph $G_m$.  Define the {\em medial weight system}, $W_m(G_m)$, using the canonical checkerboard colouring of $G_m$ as follows.  A vertex $v$ has state weights given by an ordered triple $(\alpha_v, \beta_v, \gamma_v)$, indicating the weights of the white split, black split, and crossing state, in that order. We write $(\boldsymbol\alpha, \boldsymbol\beta, \boldsymbol\gamma)$ for the set of these ordered triples, indexed, equivalently, either by the vertices of $G_m$, or by the edges of $G$.  
\end{definition}

Visually, the medial weight system $(\boldsymbol\alpha, \boldsymbol\beta, \boldsymbol\gamma)$ will assign weights to a vertex $v$ of $G_m$ as follows:
\begin{center}
\begin{tabular}{ c|ccc } 
\raisebox{5mm}{state}  &    \hspace{5mm} \includegraphics[height=12mm]{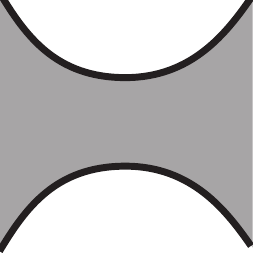} \hspace{5mm}  &   \hspace{5mm}\includegraphics[height=12mm]{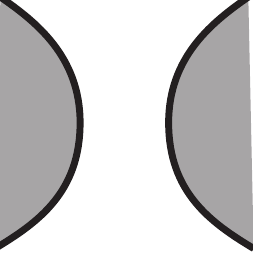} \hspace{5mm}  &   \hspace{5mm}\includegraphics[height=12mm]{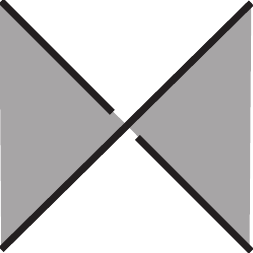}  \hspace{5mm} \\ \hline
weight & $\alpha_v$ &$\beta_v$ & $\gamma_v$ \\ 
\end{tabular}.
\end{center}

We will use the following convention:  if in a medial weight system $(\boldsymbol\alpha, \boldsymbol\beta, \boldsymbol\gamma)$ we have $\boldsymbol\alpha = \boldsymbol{k}$, where $\boldsymbol{k}\in \mathbb{N}$, we mean that  each $\alpha_v=k$, and similarly for 
$\boldsymbol\beta$ and  $\boldsymbol\gamma$. 
For example,  $(\boldsymbol\alpha, \boldsymbol\beta, \boldsymbol\gamma)=(\boldsymbol{-1}, \boldsymbol0, \boldsymbol1)$, denotes the medial weight system where $\alpha_v=-1$, $\beta_v=0$, and $\gamma_v=1$, for each $v$.

\section{The topological Penrose polynomial}\label{s.penrose}

\subsection{The Penrose polynomial of an embedded graph}\label{ss.penrose}
We now have the necessary foundations to extend the Penrose polynomial to embedded graphs.  The classical Penrose polynomial is just Definition \ref{d.Penrose} restricted to plane graphs.
\begin{definition}\label{d.Penrose}  Let  $G$ be
  an embedded graph with canonically checkerboard coloured medial graph $G_m$, let $St(G_m)$ be the set
  of states of $G_m$, and let $\mathcal{P}(G_m)$ be the set of Penrose states.  Then the {\em Penrose polynomial} is defined by 
\[
P(G;\la) := \sum\limits_{s \in St(G_m )}\omega_P (s) \la^{c(s)}   =
         \sum\limits_{s \in \mathcal{P}(G_m )}
            \left( { - 1} \right)^{cr\left( s\right)}
            \la^{c(s)}  \in \mathbb{Z}[\la]  ,  
\]
where $\omega_P (s)$ denotes the graph state weights of the medial weight system $W_P(G_m)$ defined by $(\boldsymbol{1}, \boldsymbol{0}, \boldsymbol{-1})$, where  $c(s)$ is the number of components in the graph state $s$, and
$cr(s)$ is the number of crossing vertex states in the graph state
$s$. 
\end{definition}

\begin{example}\label{ex:Penrose}
As an example, consider the graph $G$ embedded on the torus shown below.  
\[ \raisebox{7mm}{$G=$}\;\;\includegraphics[height=15mm]{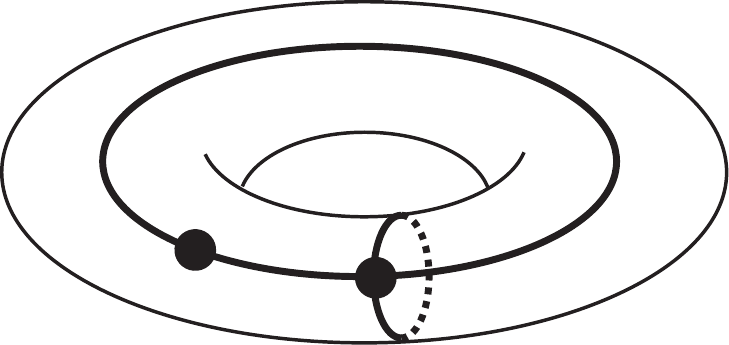}  \hspace{3cm} \raisebox{7mm}{$G_m=$}\;\;\includegraphics[height=15mm]{co6} \]

The terms of the Penrose polynomial are calculated as follows:
\smallskip

\begin{center}
\begin{tabular}{ c c c c c c c}
 \includegraphics[height=20mm]{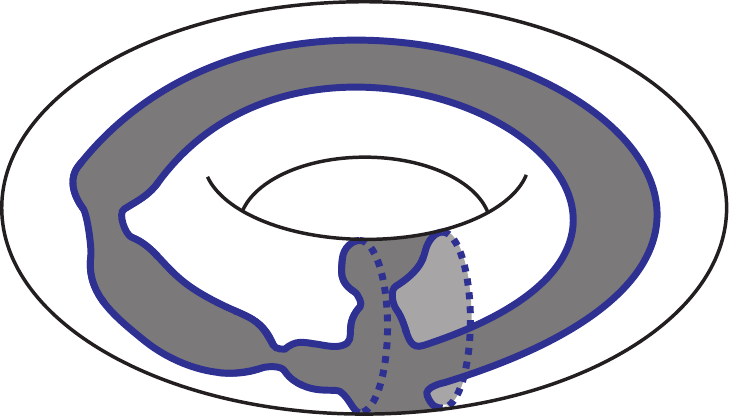} && \includegraphics[height=20mm]{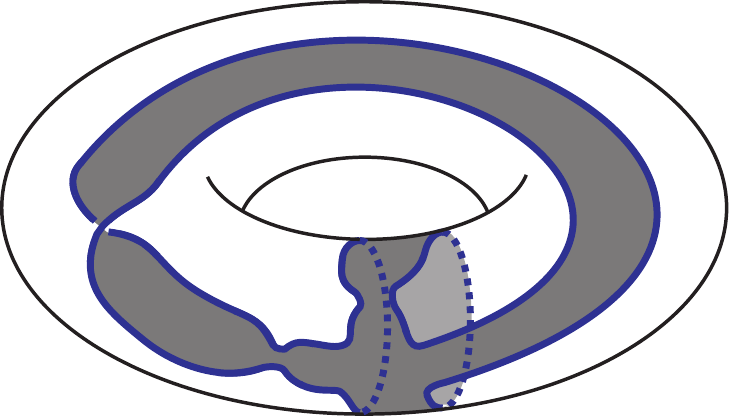}&&  \includegraphics[height=20mm]{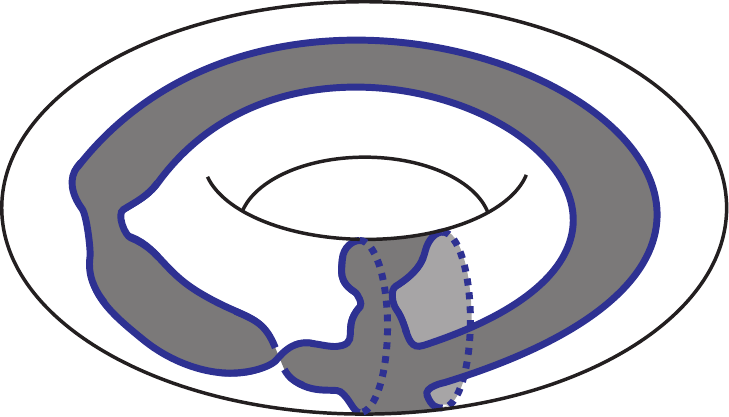}&&  \includegraphics[height=20mm]{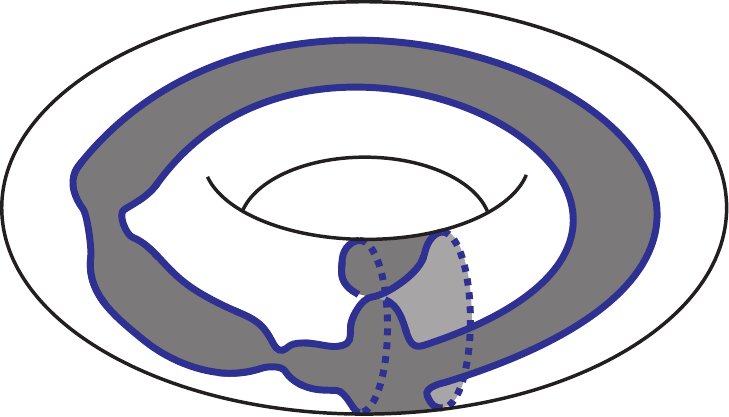} \\
 $+\lambda$ && $-\lambda$  &&$- \lambda$  && $- \lambda$ \\
 && && && \\
\end{tabular}
\begin{tabular}{ c c c c c c c}
\includegraphics[height=20mm]{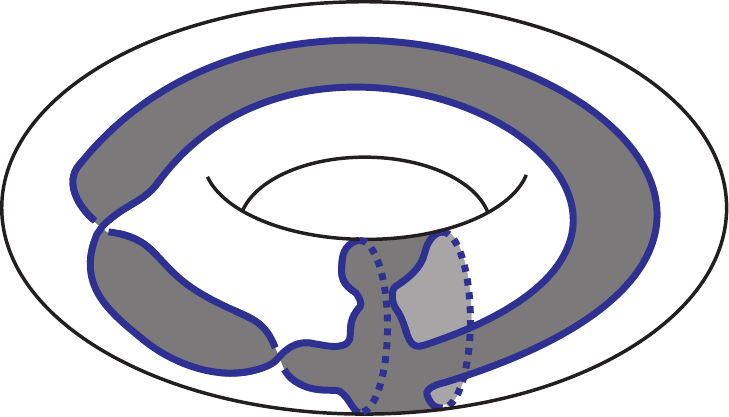}  &&\includegraphics[height=20mm]{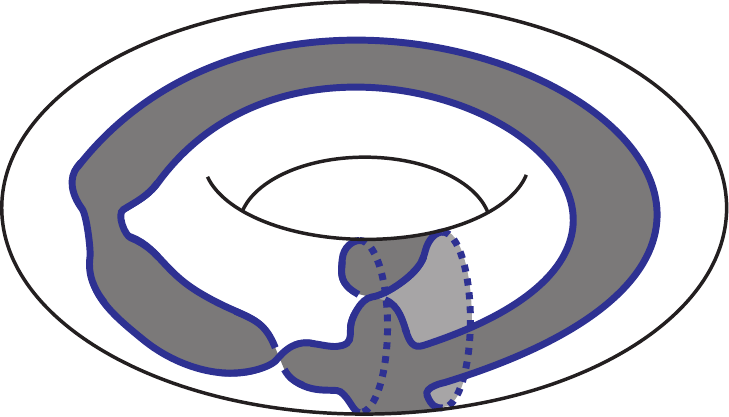} &&  \includegraphics[height=20mm]{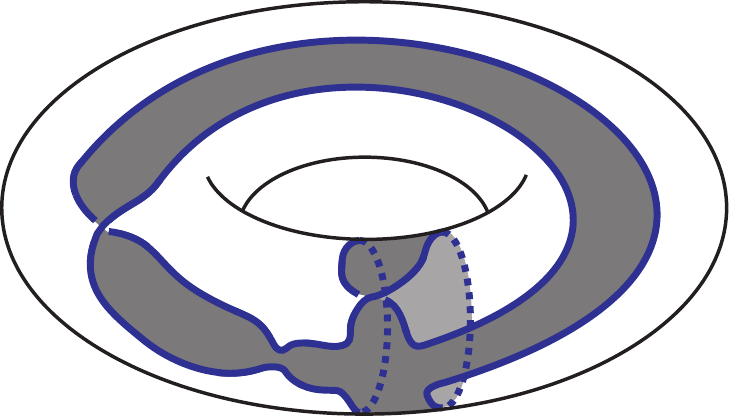}&&\includegraphics[height=20mm]{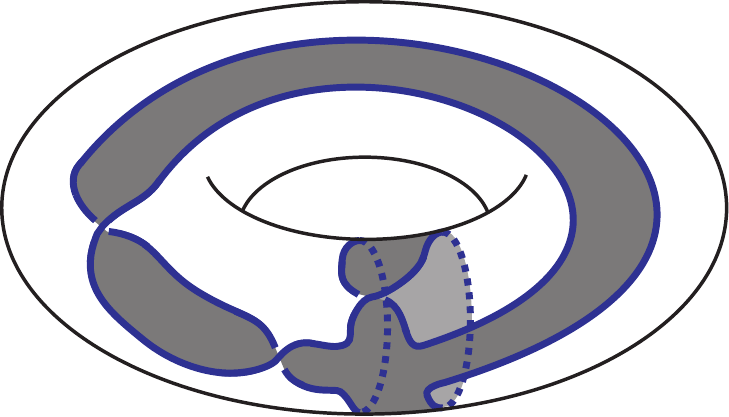}  \\
$+\lambda$ &&$+\lambda^2$  && $+\lambda^2$&& $-\lambda$ \\
 && && && \\
\end{tabular}
\end{center}
Thus, $P(G;\lambda)=  2\lambda^2-2\lambda$.

\end{example}

The
 Penrose polynomial of a plane graph may also be computed via a linear recursion
 relation (see Jaeger~\cite{Ja90}), by repeatedly  applying the skein relation 
 \begin{equation}\label{e.skP}
\quad \quad
\raisebox{-5mm}{\includegraphics[height=12mm]{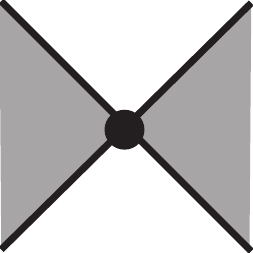}}\;\;\; = \;\; \;
\raisebox{-5mm}{\includegraphics[height=12mm]{m2}}\;\;\;
- \;\;\;\raisebox{-5mm}{\includegraphics[height=12mm]{m4}},
 \end{equation}
  to vertices of degree 4 in $G_m$, and at the end, evaluating each of the resulting cycles to $x$.  This extends \emph{mutis mutandis} to embedded graphs and the topological Penrose polynomial.

\subsection{The relation with the transition polynomial}\label{ss.trans}
There is a very useful relation, due to Jaeger~\cite{Ja90}, between the original Penrose  and transition polynomials.   Here we will show that this relation  extends to a relation between the topological Penrose polynomial  and the topological transition polynomial given in \cite{E-MMe}. Many of the results presented in this paper rely on this relation between these two topological graph polynomials.

The generalized transition polynomial, $q(G; W,t)$, of \cite{E-MS02}  extends  the  transition polynomial of Jaeger~\cite{Ja90} to arbitrary Eulerian graphs and incorporates pair and vertex state weights. For the current application, however, we will restrict $q$ to $4$-regular embedded graphs (typically medial graphs) and we will only work in the generality needed for our current application.  See \cite{E-MS02} or \cite{ES} for further details on the generalized transition polynomial. 

\begin{definition}[\cite{E-MS02}]\label{transpolydef} Let $F$ be a $4$-regular graph having weight system $W$ with values in a unitary ring $\mathcal{R}$.  Then  the {\em generalized transition polynomial} is 
\[
q(F; W,t)= \sum_{s} {\omega( s )t^{c(s)} },
\]  
where the sum  is over all graph states $s$ of $F$.  
\end{definition}  
Because we are interested in applications to the Penrose polynomial, we specialize further to embedded medial graphs and to medial weight systems and the \emph{topological transition polynomial} of \cite{E-MMe}, defined as follows.
\begin{definition}[\cite{E-MMe}]
 Let $G$ be an embedded graph with canonically checkerboard coloured embedded medial graph $G_m$, and let $W_m(G_m)=(\boldsymbol\alpha, \boldsymbol\beta, \boldsymbol\gamma)$ be a medial weight system. Then the \emph{topological transition polynomial} of $G$ is:
\[
Q(G, (\boldsymbol\alpha, \boldsymbol\beta, \boldsymbol\gamma), t) :=q(G_m; W_m,t).
\] 
\end{definition}

We will also use the following recursive formulation of the topological transition polynomial, which follows from its being a specialization of the generalized transition polynomial, which has such a recursion (see \cite{E-MS02}).
We use the notation that if $v$ is a vertex of $G_m$, then  $(G_m)_{bl(v)}$, $(G_m)_{wh(v)}$, and $(G_m)_{cr(v)}$ denote the graphs obtained by  taking a black, white, or crossing state, respectively,  at the vertex $v$.
\begin{proposition}\label{p.recQ}
The topological transition polynomial may be computed by repeatedly applying  the following linear recursion relation at each $v\in V(G_m)$, and, when there are no more vertices of degree 4 to apply it to, evaluating each of the resulting closed curves to an independent variable $t$: 
\[
q(G_m, W_m,t)= \alpha_v q((G_m)_{wh(v)}, W_m,t)+ \beta_v q((G_m)_{bl(v)}, W_m,t)+ \gamma_v q((G_m)_{cr(v)}, W_m,t).
\]
\end{proposition}

The recursion relation in Proposition~\ref{p.recQ} has the following pictorial presentation:
\begin{equation}\label{e.skT} \includegraphics[height=14mm]{m1v}\;\; \raisebox{7mm}{$=$} \;\; \raisebox{7mm}{$\alpha_v$} \; 
\includegraphics[height=14mm]{m2}
\raisebox{7mm}{$\;+ \;\; \beta_v$} \;\; 
\includegraphics[height=14mm]{m3}
\raisebox{7mm}{$\;+ \;\;\; \gamma_v$}\;\;\includegraphics[height=14mm]{m4}\raisebox{7mm}{.}
 \end{equation}

A comparison between the skein relation for the Penrose polynomial in Equation~\eqref{e.skP}, and the skein relation for the topological transition polynomial in Equation~\eqref{e.skT}
 immediately makes clear how the Penrose polynomial is related to the topological transition polynomial: 
\begin{proposition}\label{p.petr}
Let $G$ be an embedded graph and $G_m$ be its canonically checkerboard coloured medial graph.  Then   
\[   P(G;\la)= q(G_m; W_P, \la) =Q(G; (\boldsymbol1, \boldsymbol0,\boldsymbol{-1}), \la) ,\]
where $Q$ is the topological transition polynomial.
\end{proposition}

\subsection{Some distinctions between the classical and topological Penrose polynomials}\label{distinct}

The Penrose polynomial of a plane graph is known to satisfy several combinatorial identities and has numerous connections with graph colouring. It is natural to ask which of these properties hold for arbitrary embedded graphs. 
We begin by observing that many of the basic properties of the Penrose polynomial of a plane graph given by Penrose in  \cite{Pen71}, Jaeger in \cite{Ja90}, and Aigner in  \cite{Ai97} do not hold for non-plane graphs.
For example, the following properties for plane graphs provided by Aigner in \cite{Ai97} do not hold for embedded graphs in general, although we will see in Section \ref{s.cpp} that some of them do hold for orientable checkerboard colourable graphs. 
  Some small counterexamples to  non-plane extensions of these properties are provided by: $F_1 = $\raisebox{-2mm}{\includegraphics[height=6mm]{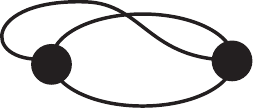}} with 
  $P(F_1; \la)=-\la^3+3\la^2-2\la$; 
  by  $F_2 = $\raisebox{-2mm}{\includegraphics[height=6mm]{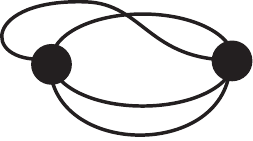}} with 
  $P(F_2; \la)=\la^3-\la$; and 
  by  $F_3 = $\raisebox{-2mm}{\includegraphics[height=6mm]{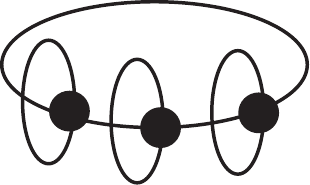}} with 
  $P(F_3; \la)=4(\la^2-\la)$.   
  $F_1$ gives counter examples to Items  \ref{faces},  \ref{3colour}, 
  while $F_2$ gives counter examples to Items \ref{verts}, \ref{edges}, \ref{twoface} and \ref{faces}, 
  and $F_3$ to Items  \ref{edges}, and \ref{notwoface}.  For Item~\ref{bridge}, if $G$ is an embedded graph that contains a bridge, the $P(G;\la)=0$, but Example~\ref{hemi} will give an example of a bridgeless, non-plane embedded graph with trivial Penrose polynomial.  
\begin{enumerate}
 \item \label {verts} If $G$ is plane and Eulerian then $P(G;2) =2^{v(G)}$.
\item \label{edges} If $G$ is plane and Eulerian then $P(G;-1) =(-1)^{f(G)}2^{e(G)}$.
\item  \label{twoface}If $G$ is plane and edges $e,f$ are both on the same two faces, then $P(G; \la)=2P(G/e;\la)$.  
\item \label{notwoface} If $G$ is plane, 2-connected, and has no two faces sharing more than one edge, then the leading term of $P(G;\la)$ is $1$. 
\item  \label{faces} If $G$ is plane then the degree of $P(G;\la)$ is the number of faces of $G$. 
\item \label{3colour} If $G$ is plane, cubic and connected, then $P(G;3)$ counts the number of edge 3-colourings. 
\item \label{bridge} If $G$ is a plane graph, then $G$ contains a bridge if and only if $P(G;\la)=0$.
\end{enumerate}

The Penrose polynomial also  has  the following properties for plane graphs:
 if $G$ is Eulerian, the coefficients of $P(G;\la)$ alternate in sign; and  if $G$ is non-Eulerian,  $|P(G;-1)| <2^{e(G)}$. We would like to know if these properties hold more generally.  A natural extension is to characterize the embedded graphs for which the  properties in the above list do hold or give more inclusive combinatorial interpretations for these evaluations of the Penrose polynomial.

\section{Twisted duality}\label{s.td}
In this section we provide an overview of twisted duality and the ribbon group action which were introduced by the authors in \cite{E-MMe} to unravel the connections between graphs and their medial graphs, and  to solve isomorphism problems for $4$-regular graphs. The ribbon group action and the results described  in this section provide the tools that we use to determine properties the Penrose polynomial. 

Unless otherwise stated, all of the results in this section appear in \cite{E-MMe}, and their proofs and further details  may be found therein.

\subsection{Twisted duals and the ribbon group}\label{ss.td}
In this section we provide an overview of twisted duals and the ribbon group action. The ribbon group action is a far reaching generalization of the idea of the geometric dual  of an embedded graph. The importance of the ribbon group lies in that it provides new connections, and a new understanding of the relationship between graphs and their medial graphs. For example, \cite{E-MMe} shows that two medial graphs $G_m$ and $H_m$ are isomorphic as abstract graphs if and only if $G$ and $H$ are twisted duals. Here we will focus on the way that the ribbon group interacts with graph polynomials, such as the Penrose polynomial, that can be defined through medial graphs.

Let $\calG_{(n)}$ denote the set of equivalence classes of embedded graphs with exactly $n$ edges.  (Recall that the equivalence relation here is generated by homeomorphism of surfaces.)
 Furthermore, let
\[\calG_{or (n)} = \left\{  (G,\ell)| G \in \calG_{(n)} \text{ and } \ell \text{ is a linear ordering of the edges} \right\}\] be the set of equivalence classes of embedded graphs with  exactly $n$ ordered edges. As is standard, we will abuse notation and identify a representative of an equivalence class with the class itself.

We will define two operations, a half-twist and a dual,  that act on a specified edge of an embedded graph. These operations generate an action, called the ribbon group action, of the group $\fS^n$ on the set $\calG_{or (n)}$, where $\fS$ is the symmetric group of degree three. 

\begin{definition}\label{def.ops}
Let $(G,\ell) \in \calG_{or(n)}$ and suppose $e_i$ is the $i^\text{th}$ edge in the ordering $\ell$. Also, suppose $G$ is given in terms of its arrow presentation, so $e_i$ is a label of a pair of arrows.  

The \emph{half-twist of the $i^{\text{th}}$ edge} is $(\tau, i)(G,\ell)=(H,\ell)$ where $H$ is obtained from $G$  by reversing the direction of exactly one of the $e_i$-labelled arrows of the arrow presentation. $H$ inherits its edge order $\ell$ in the natural way from $G$.  

The \emph{dual with respect to the $i^{\text{th}}$ edge} is $(\delta,i) (G,\ell) = (H,\ell)$, where $H$ is obtained from $G$ as follows.  Suppose $A$ and $B$ are the two arrows labelled $e_i$ in the arrow presentation of $G$.  Draw a line segment with an arrow on it directed from the the head of $A$ to the tail of $B$, and a line segment with an arrow on it directed from the head of $B$ to the tail of $A$.  Label both of these arrows $e_i$, and delete $A$ and $B$ with the arcs containing them. The line segments with their arrows become arcs of a new circle (or circles) in the arrow presentation of $H$.  As with the twist, $H$ inherits its edge order $\ell$ from $G$.

\end{definition}

The actions of $(\tau,i)$ and $(\delta,i)$ on an arrow presentation are illustrated thus:
\[  \tau\left( \;\; \raisebox{-4mm}{\includegraphics[height=10mm]{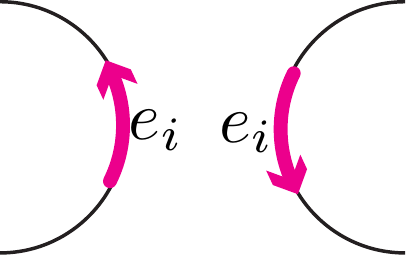}}\;\; \right) \; = \; \;\raisebox{-4mm}{\includegraphics[height=10mm]{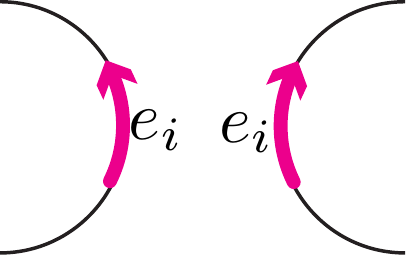}} 
\hspace{1cm} \raisebox{0mm}{\text{and}}\hspace{1cm}
\delta\left( \;\; \raisebox{-4mm}{\includegraphics[height=10mm]{a1ei}} \;\; \right) \; = \; \; \raisebox{-4mm}{\includegraphics[height=10mm]{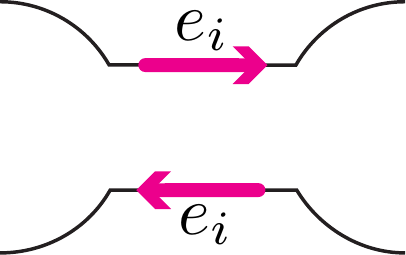}}. \]
Examples of the action of $(\tau,i)$ and $(\delta,i)$ can be found in Example~\ref{e.pd} below.

It was shown in \cite{E-MMe} that, for each $i$,  we have ${(\tau^2, i)}(G,\ell)={ (\delta^2,i)}(G,\ell)=  ({(\tau \delta)}^3, i)(G,\ell)=1(G,\ell)$.  Consequently,  the group \[ \fG := \langle  \delta, \tau   \; |\;   \delta^2, \tau^2, (\tau \delta)^3   \rangle ,\]
which is isomorphic to the symmetric group of order three, acts on the set $\calG_{or(n)}$. 

This action readily extends to a group action of $\fS^n$ on  $\calG_{or (n)}$ by allowing the twists and partial duals to act on any subset of edges, not just a single distinguished edge. We 
call $\fG^{n}$ the \emph{ribbon group for $n$ edges} and define the \emph{ribbon group action} to be the action of  $\fG^{n}$ on $\calG_{or(n)}$ given by:
\begin{eqnarray*} ( \g_1, \g_2,\g_3, \ldots , \g_n )(G,\ell) &= (\g_n,n)((\g_{n-1}, n-1) \ldots ((\g_2,2)(( \g_1,1)(G,\ell))) \ldots) \\
&=  ( (\g_n,n)\circ(\g_{n-1}, n-1)\circ \ldots \circ (\g_2,2) \circ ( \g_1,1)   ) (G,\ell) 
,\end{eqnarray*}
where  $\g_i \in \fG$ for all $i$.

We make the simple but important observation that these operations when applied to different edges commute:
if $i \neq j$ and $\g, \h \in \{ \tau, \delta \}$, then $(\g,i) ((\h,j)(G, \ell)) = (\h,j)((\g,i)(G,\ell))$.
However, in general,  $(\g, i)$ and $(\h,i)$ do  not commute.

We can now define twisted duals to be embedded graphs that are related under the ribbon group action.
\begin{definition}
Two embedded graphs with linearly ordered edges $(G,\ell)$ and $(G', \ell')$ are said to {\em
twisted duals} if (the equivalence class of)   $(G', \ell')$ is in the orbit of (the equivalence class of) $(G,\ell)$ under the ribbon group action.

Two embedded graphs $G$ and $G'$ are said to be {\em twisted duals} if $(G,\ell)$ and $(G', \ell')$ are twisted duals for some edge orderings $\ell$ and $\ell'$ of $G$ and $G'$ respectively.
\end{definition}

As an example, all of the embedded graphs shown in Example~\ref{e.pd} are twisted duals. 

\medskip

In this paper, we will focus on twisted duality for embedded graphs without edge ordering. In \cite{E-MMe}, it was shown that every twisted dual of an embedded graph $G$ can be written in the form $G^{\prod{\g_i(A_i)}}$. For our purposes, this notation  is particularly efficient and we will now give a quick exposition of it. 
In order to introduce the notation $G^{\prod{\g_i(A_i)}}$, we need to first introduce the notation $G^{\g(A)}$. To do this, 
suppose that $G \in \calG_{(n)}$, $A \subseteq E(G)$ and  $\g \in \fG$.  Then  $G^{\g(A)}$ is defined by letting $\ell$ be an arbitrary ordering $(e_1, \ldots, e_n)$ of the edges of $G$; and $\vg_{A} :=(\epsilon_1, \ldots, \epsilon_n) \in \fG^n$, where $\epsilon_i = \g$ if $e_i \in A$ and $\epsilon_i=1$ otherwise. With this, we can then define 
\[
G^{\g(A)}:=\vg_{A}(G,\ell).
\] 
If, in addition,  $ B \subseteq E(G)$ and  $\h \in \fG$ then we set
$
G^{\g(A)\h(B)}:=(G^{\g(A)})^{\h(B)}$, and
$G^{\g\h(A)}:=G^{\h(A)\g(A)}$.

With this notation, it can be shown (see Proposition~3.7 of \cite{E-MMe}) that every twisted dual admits a unique expression of the form 
\begin{equation}\label{e.ga}G^{\prod^6_{i=1}{\g_i(A_i)}}
,\end{equation}
where the $A_i$ partition $E(G)$, and where $\g_1=1, \g_2=\tau , \g_3=\delta , \g_4=\tau \delta , \g_5= \delta \tau$, and  $\g_6 = \tau \delta \tau \in \fG$. 

To illustrate the ribbon group action, we recall the following example from \cite{E-MMe}.
 \begin{example}\label{e.pd}
 If $G$ is an embedded graph with $E(G)=\{e_1,e_2\}$, with the order $(e_1,e_2)$, represented as an arrow presentation and as a ribbon graph shown below,
\[
\raisebox{8mm}{$G =$}   \;\;\raisebox{1mm}{\includegraphics[width=30mm]{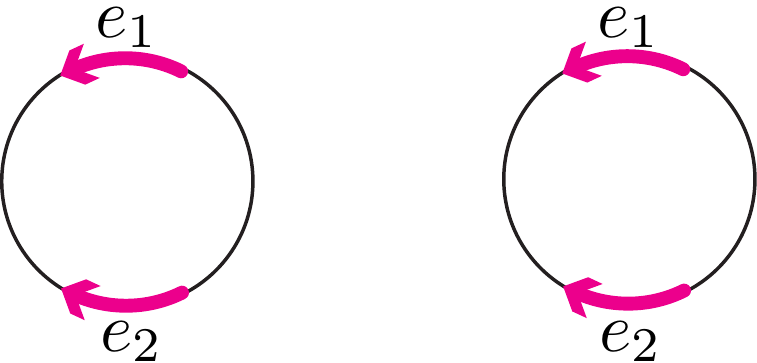}} \; \;\raisebox{8mm}{$=$} \;  \; \raisebox{0mm}{\includegraphics[width=34mm]{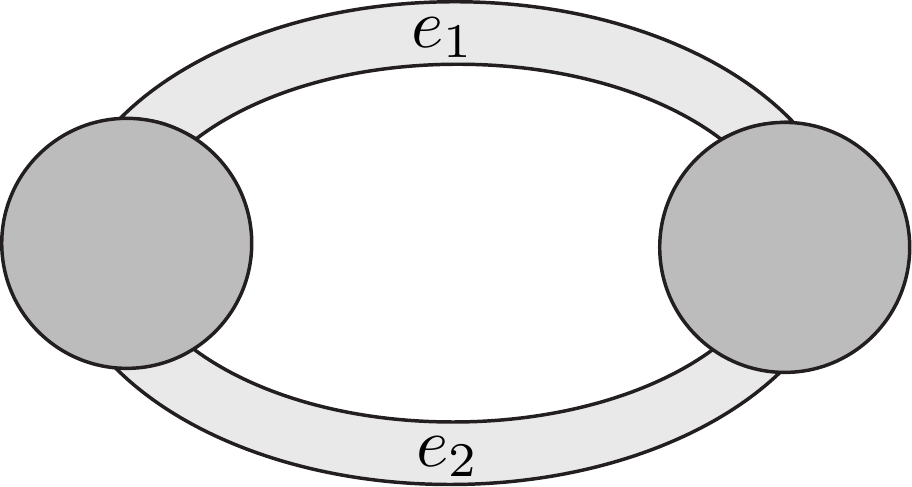}} \;\raisebox{8mm}{,}
\]
then we have
\[
\raisebox{8mm}{$(\tau, 1)(G) =G^{\tau(e_1)} =$}   \;\;\raisebox{1mm}{\includegraphics[width=30mm]{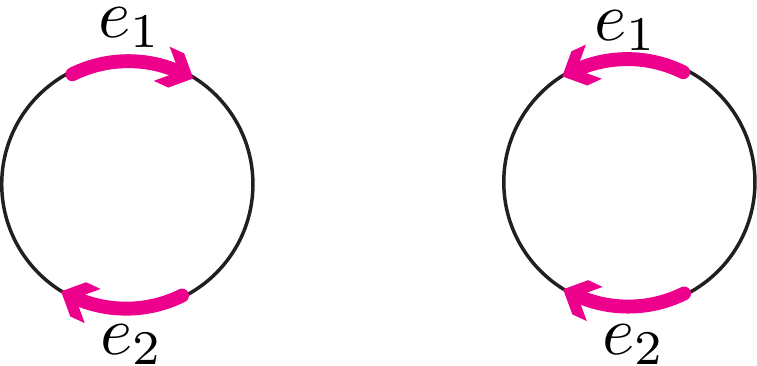}} \; \;\raisebox{8mm}{$=$} \;  \; \raisebox{0mm}{\includegraphics[width=34mm]{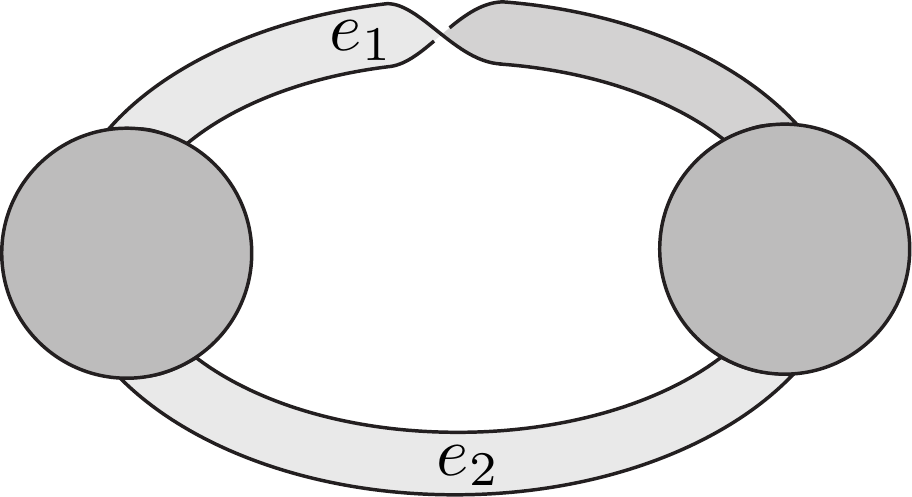}} \;\raisebox{8mm}{,}
\]
and 
\[
\raisebox{8mm}{$(\delta , 1)(G) =G^{\delta(e_1)} =$}   \;\;\raisebox{4mm}{\includegraphics[width=30mm]{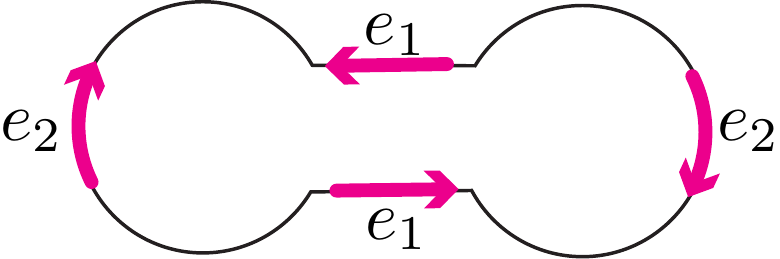}} \; \;\raisebox{8mm}{$=$} \;  \; \raisebox{1mm}{\includegraphics[width=24mm]{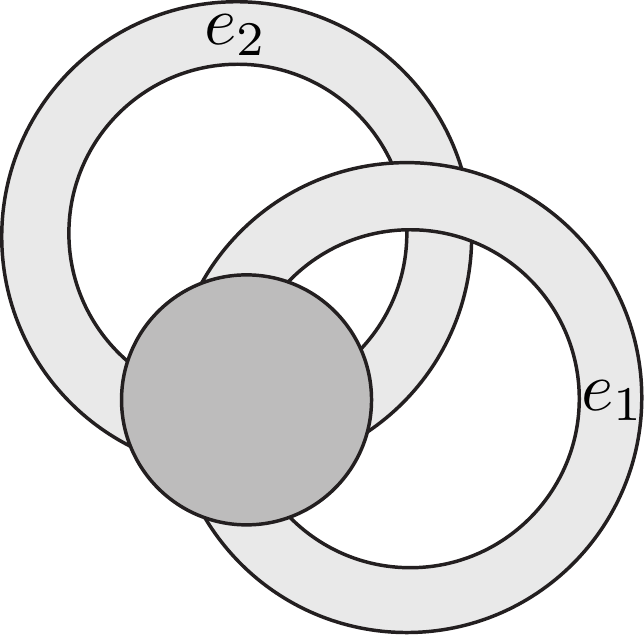}} \;\raisebox{8mm}{.}
\]

The full orbit of $G$ is 
\[\includegraphics[height=4cm]{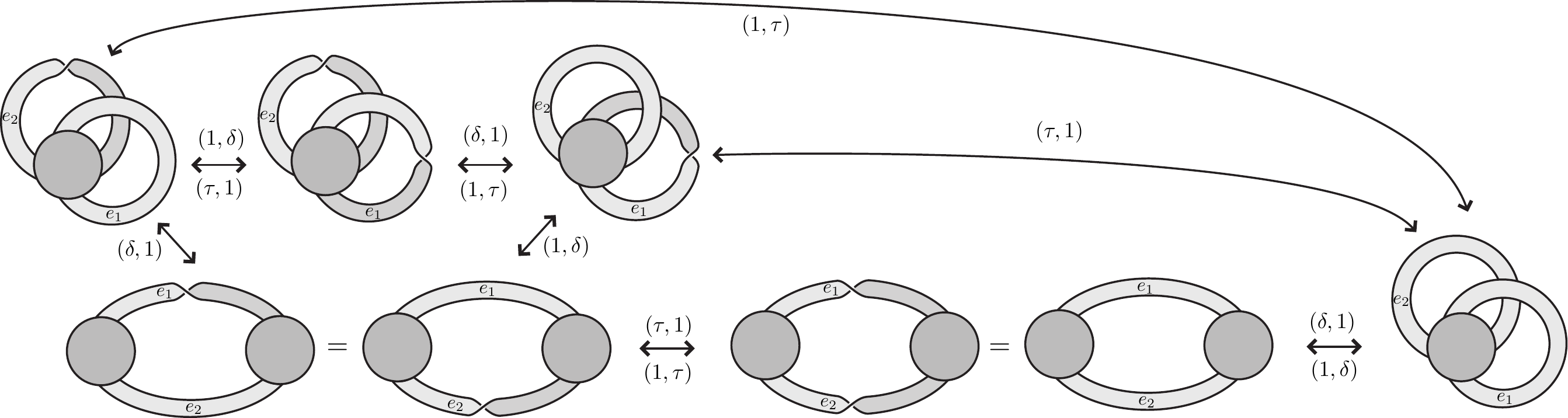}\]
\end{example}

\begin{example}
If $G= \raisebox{-8mm}{\includegraphics[height=2.2cm]{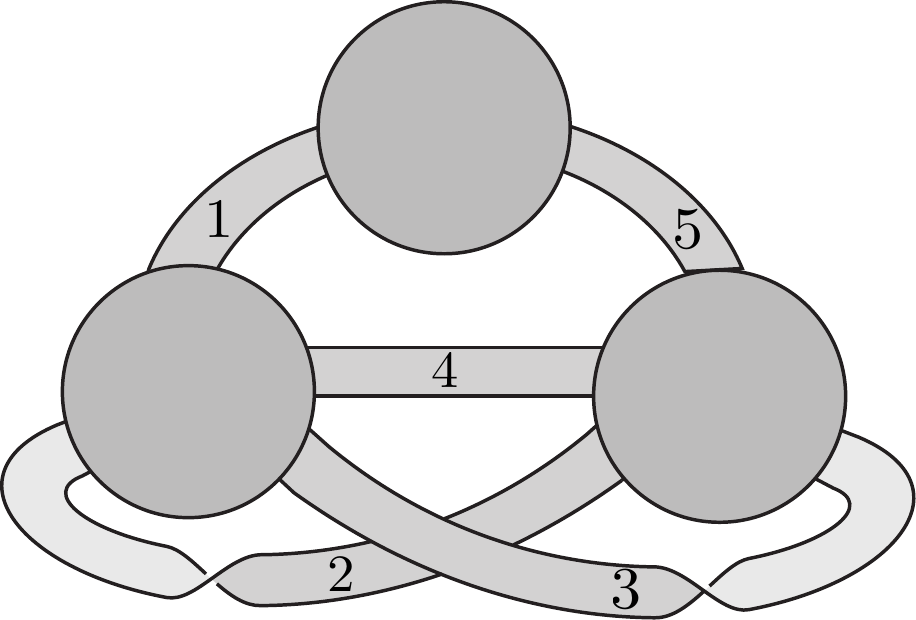}}$, then 
\[G^{\tau(3)\delta(4)\delta\tau(2)   }= \raisebox{-7mm}{\includegraphics[height=1.8cm]{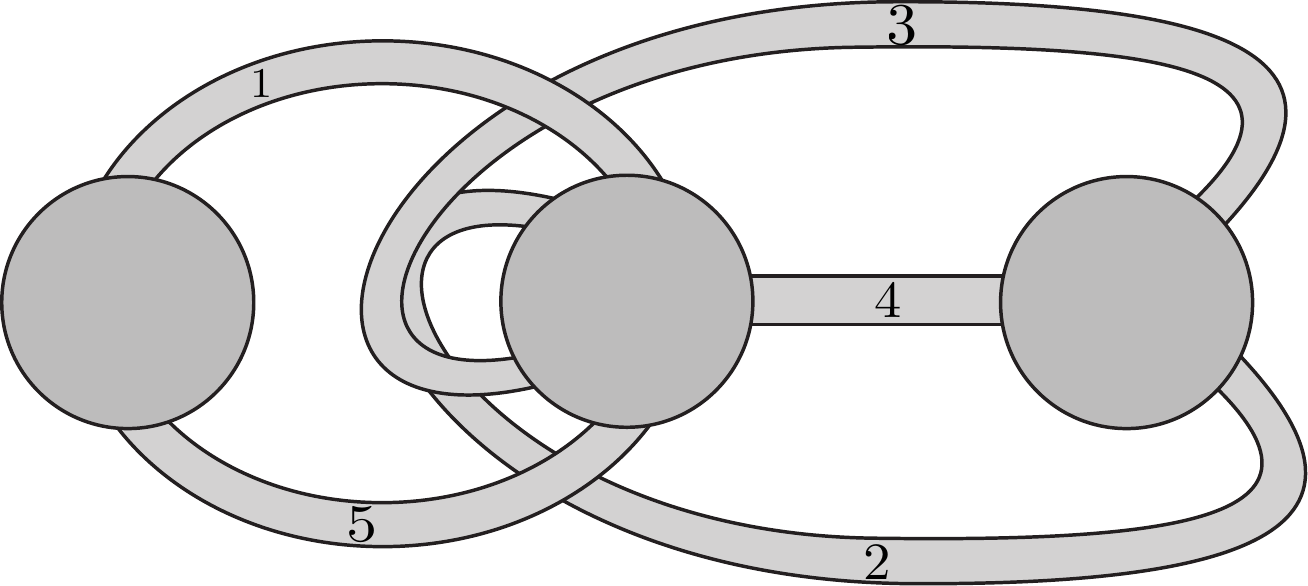}} \quad\text{and} \quad
       G^{  \tau(3)\delta(1,2)\tau\delta(5)\tau\delta\tau(4)  }= \raisebox{-11mm}{\includegraphics[height=2.4cm]{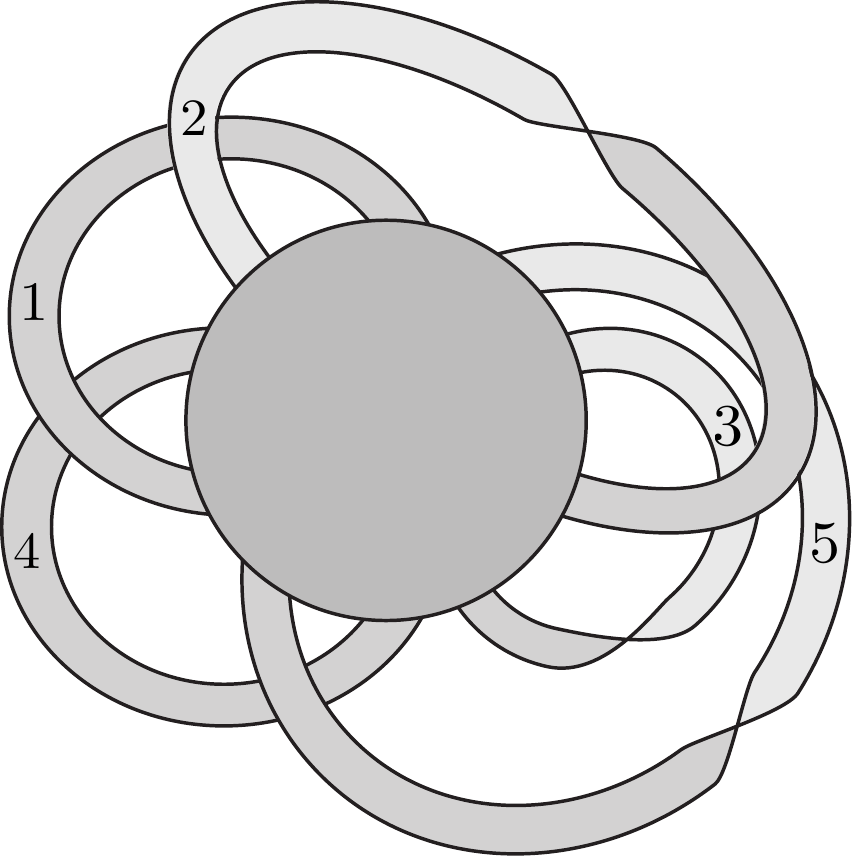}}\;.\]
\end{example}

Other important forms of duality appear as actions of subgroups of of the ribbon group on the set of embedded graphs. These dualities include geometric duality, Petriality, and Chmutov's partial duality from \cite{Ch1}. These connections were discussed in detail in \cite{E-MMe}, and we summarize them below.
\begin{propdef}\label{p.otherduals}
Let $G$ be an embedded graph and $A\subseteq E(G)$. Then
\begin{enumerate}
\item $G^{\delta(E(G))}=G^*$, the {\em geometric dual} of $G$;
\item $G^{\delta(A)}$ is the {\em partial dual} of $G$ with respect to $A$;
\item $G^{\tau(E(G))}$ is the {\em Petrial} or {\em Petrie dual} of $G$;
\item $G^{\tau(A)}$ is the {\em partial Petrial} of $G$ with respect to $A$.
\end{enumerate}
\end{propdef}

\subsection{The action of the ribbon group on the transition polynomial}\label{ss.tdtrans}
We will be especially interested in the behaviour of the Penrose polynomial under the operations of twisted duality.  In order to study how the Penrose polynomial of an embedded graph and its twisted duals are related, we  consider the action of the ribbon group on the more general topological transition polynomial. The effect of twisted duality on the Penrose polynomial can then be determined by applying Proposition~\ref{p.petr}. In this subsection we review the action of the ribbon group on the topological transition polynomial that was introduced in \cite{E-MMe}.

The group $\fG= \langle  \delta, \tau    \; |\;   \delta^2, \tau^2, (\tau\delta)^3   \rangle$ is isomorphic to $\fS_3$ via an isomorphism $\eta$ defined by
\[ 
\eta: \tau \mapsto (1\; 3) \quad \text{ and } \quad \eta: \delta \mapsto (1\; 2).  
\]
Furthermore, the symmetric group $\fS_3$  acts on the ordered triple, $(\al_v,\bet_v,\ga_v)$, of the weight system at a vertex by permutation. 
This action by $\fS_3$ on  weight systems at a vertex can be extended to an action of $\fS_3^n$ on medial weight systems of medial graphs with $n$ linearly ordered vertices. Since we will only need  the order independent analogue of this action, we will focus on it here.
This action allows us to use the ribbon group to modify the medial weight system of an embedded medial graph.  

To define this action, let $G_m$ be a canonically checkerboard coloured embedded medial graph of an embedded graph $G$ with medial weight system $W_m$ (or equivalently $(\bal, \bbe,\bga)$). Suppose further that $\Gamma = \prod^6_{i=1}{\g_i(A_i)}$ where the $A_i$'s partition $E(G)$, and the $\g_i$'s are the six elements of $\fG$ (as in Equation~\eqref{e.ga}).
Then we define ${W_m}^{\Gamma}$ (or $(\bal, \bbe,\bga)^{\Gamma}$), the \emph{weight system permuted by $\Gamma$}, to be the ordered triple of the weight system at a vertex $v_e$ given by $\eta ( \g_i) (\alpha_{v_e},\beta_{v_e},\gamma_{v_e})$ when $e \in A_i$.

\medskip

 In \cite{E-MMe}, it was shown that this action of the ribbon group on a medial weight system is compatible with the action of the ribbon group on embedded graphs. This compatibility is given in Theorem~\ref{t.qsd} as a twisted duality relation for the topological transition polynomial.
  This twisted duality relation says that the topological transition polynomial of the medial graph of $G$ is the same as that of the medial graph of any of the twisted duals, provided the weight system is appropriately permuted.  We will   apply this twisted duality relation in the subsequent sections to derive new properties for the Penrose polynomial.

\begin{theorem}\label{t.qsd}
Let $G$ be an embedded graph with embedded medial graph $G_m$, and let $\Gamma = \prod^6_{i=1}{\g_i(A_i)}$ where the $A_i$'s partition $E(G)$, and the $\g_i$'s are the six elements of $\fG$. Then,
\[ 
q\left(G_m; W_m, t\right) = q\left(G_m^{\Gamma}; W_m^{\Gamma}, t\right),  
\] 
\noindent
or equivalently,
\[
Q(G; (\boldsymbol\alpha, \boldsymbol\beta, \boldsymbol\gamma), t)=Q(G^{\Gamma}, (\boldsymbol\alpha, \boldsymbol\beta, \boldsymbol\gamma)^{\Gamma}, t).
\]
\end{theorem}

We note that this twisted duality identity for the transition polynomial unifies a number of duality relations for several important graph polynomials. These relations include the  well-known identity $T(G;x,y)=T(G^*;y,x)$ for the Tutte polynomial; the identity $Q(G; (\boldsymbol\alpha, \boldsymbol\beta, \boldsymbol\gamma), t)=Q(G^*; (\boldsymbol\beta, \boldsymbol\alpha, \boldsymbol\gamma), t)$, from \cite{ES}, for the topological transition polynomial; the identities for  Bollob\'as and Riordan's topological Tutte polynomial under duality from \cite{ES, Mo1}, and under partial duality from \cite{Ch1, E-MMd, Mo3, VT}.

\subsection{Deletion, contraction, and the transition polynomial}\label{ss.delcon}
If $H$ is an abstract ({\em i.e.} non-embedded) graph and $e$ is any edge of $H$, then the contraction $H/e$ of $H$ along $e$ will always result in an abstract graph, and so contraction is a well defined operation on abstract graphs. However, for graphs embedded in surfaces, contraction is a more subtle operation. Suppose that a graph $G$ is   cellularly embedded in a surface $\Sigma$. If $e$ is not a loop,  then $G/e$ is defined as the embedded graph obtained taking the quotient $\Sigma/\{e\}$.  Contraction then works as expected for non-loops. Moreover, this definition of contraction is compatible with the contraction of a non-loop in a ribbon graph (if $G$ is a ribbon graph, then $G/e$ is the ribbon graph obtained by contracting the edge disc $e$).
   However, difficulties arise when the edge $e$ is a loop. In this case, the quotient space  $\Sigma/\{e\}$ need not be a surface.  So with this definition of contraction, the contraction of a loop in a cellularly embedded graph can result in a graph embedded in a pseudo-surface. Furthermore, if we view $G$ as a ribbon graph, and if $e$ is a loop incident to $v$, then forming the quotient  of the ribbon graph $G$ by the disc $e$ will take the vertex $v$ of $G$ to an annulus rather than a disc, and the quotient is not a ribbon graph.

In \cite{BR2}, Bollob\'as and Riordan defined a way to contract edges of a ribbon graph in a way that was compatible with the usual contraction of non-loop edges, and with the deletion-contraction relations of their topological Tutte polynomial. Chmutov, in \cite{Ch1}, showed that Bollob\'as and Riordan's definition of contraction can be cleanly expressed in terms of partial duality. We take Chmutov's expression to be our definition of the contraction of any edge of an embedded graph.
\begin{definition}\label{d.contract}
Let $G$ be an embedded graph and $e$ be {\em any} edge of $G$, then 
\[ G/e:=G^{\delta(e)} - e  .\]
\end{definition}  

\begin{example}
Let $G$ be the embedded graph consisting of a single vertex and a loop $e$. Then $G/e$ is the embedded graph consisting of two isolated vertices.

As a second example, if $G=\raisebox{-7mm}{\includegraphics[height=16mm]{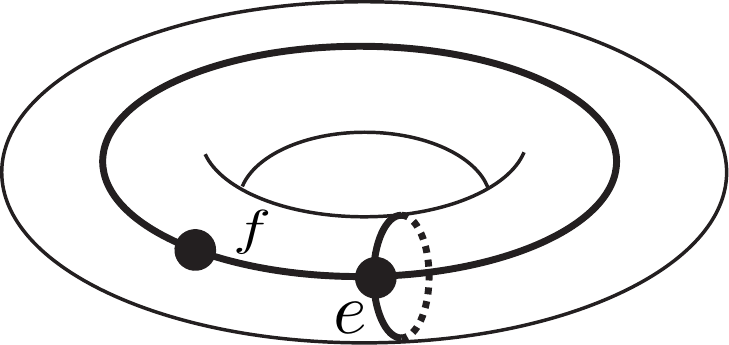}}$, then 
$G/f=\raisebox{-7mm}{\includegraphics[height=16mm]{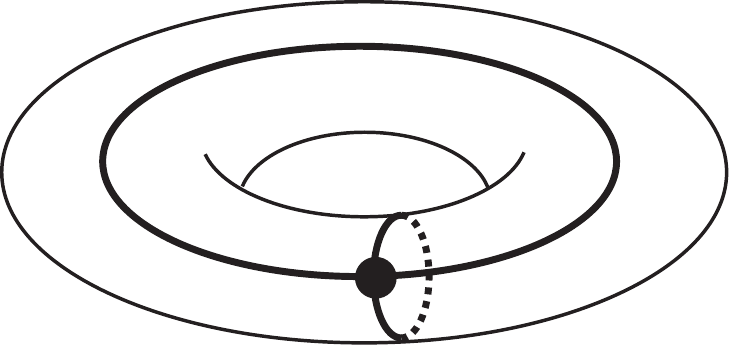}}$, and \mbox{$G/e=\raisebox{-0mm}{\includegraphics[width=20mm]{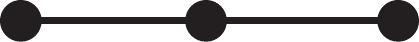}}\;\subset S^2$}.
\end{example}

This definition of contraction allowed the authors to find, in \cite{E-MMe},  a three term deletion-contraction relation for the topological transition polynomial. This deletion-contraction relation, which is stated in the following theorem, will be important in our study of the Penrose polynomial. 
\begin{theorem}\label{t.transdl} Let $G$ be an embedded graph and $e \in E(G)$. Then 
\[
Q(G; (\boldsymbol\alpha, \boldsymbol\beta, \boldsymbol\gamma), t)=\alpha_e Q(G/e; (\boldsymbol\alpha, \boldsymbol\beta, \boldsymbol\gamma), t)+\beta_e Q(G-e; (\boldsymbol\alpha, \boldsymbol\beta, \boldsymbol\gamma), t)+\gamma_e Q(G^{\tau(e)}/e; (\boldsymbol\alpha, \boldsymbol\beta, \boldsymbol\gamma), t),
\]
where the weight systems on the right-hand side are the weight systems on  $G/e$, $G-e$ or $G^{\tau(e)}/e$ induced by  $(\boldsymbol\alpha, \boldsymbol\beta, \boldsymbol\gamma)$. 

\end{theorem}

\begin{corollary}\label{c.trans delete contract} Let $G$ be an embedded graph, $e \in E(G)$ and   $\g \in \fG$. Let $(\bal,\bbe,\bga)^{\g(e)}=(\bal',\bbe',\bga')$, so that , $(\alpha'_e, \beta'_e, \gamma'_e) = \eta(\g)(\alpha_e, \beta_e, \gamma_e)$ where the permutation $  \eta(\g)$ is defined in Subsection~\ref{ss.tdtrans}. Then
\[
Q(G; (\boldsymbol\alpha, \boldsymbol\beta, \boldsymbol\gamma), t)=\alpha_e' Q(G^{\g(e)}/e; (\boldsymbol\alpha, \boldsymbol\beta, \boldsymbol\gamma), t)+\beta_e' Q(G^{\g(e)}-e; (\boldsymbol\alpha, \boldsymbol\beta, \boldsymbol\gamma), t)+\gamma_e' Q(G^{\tau\g(e)}/e; (\boldsymbol\alpha, \boldsymbol\beta, \boldsymbol\gamma), t),
\]
where the weight systems on the right-hand side are the weight systems on   $G^{\g(e)}/e$, $G^{\g(e)}-e$ or $G^{\tau\g(e)}/e$ induced by  $(\boldsymbol\alpha, \boldsymbol\beta, \boldsymbol\gamma)$.
\end{corollary}

\section{Twisted duality and identities for the Penrose polynomial}\label{s.tdpenrose}

One of the major advantages of considering the topological Penrose polynomial is that it satisfies various identities that are not realizable in terms of plane graphs.  Many of these identities arise by considering twisted duality. For example, we see in this context that, unlike the classical Penrose polynomial of a plane graph, the topological Penrose polynomial has deletion-contraction reductions similar to those for the Tutte polynomial.

\subsection{Some non-plane identities for the Penrose polynomial}

The following proposition provides a reformulation of the Penrose polynomial that will be useful later. This proposition expresses the Penrose polynomial $P(G;\lambda)$ as a sum over all of the partial Petrials,  $G^{\tau (A)}$, of $G$.
 \begin{proposition}\label{l.peng}
 Let $G$ be a ribbon graph, then 
 \[    P(G;\la)=\sum_{A \subseteq E(G)} (-1)^{|A|} \lambda^{f(G^{\tau (A)})},\]
recalling that $f(G)$ is the number of faces of $G$.
 \end{proposition} 
\begin{proof}


Let $G_m$ be naturally embedded in the same surface as $G$. To each Penrose state $s$ of $G_m$ associate a set $A_s\subseteq E(G)$ by including an edge $e$ in $A_s$ if and only if the vertex state $s_e$ at the vertex $v_e$ of $G_m$ corresponding to $e$ is a crossing. This gives a bijection between the set of Penrose states and the set of subsets of $E(G)$. Clearly $cr(s) =|A_s|$. To prove the lemma we must show that $c(s)=f(G^{\tau (A_s)})$. To see why this is the case, let $a,b,c,d$ be the corners of $e$. If $s_e$ is a crossing that forms arcs $(ac)$ and $(bd)$ along $e$, then  in G$^{\tau (A_s)}$, the boundary of the half-twisted edge $e$ also forms arcs $(ac)$ and $(bd)$. Similarly, 
 if $s_e$ is a white split that forms arcs $(ad)$ and $(bc)$ along $e$, then in G$^{\tau (A_s)}$, the boundary of the edge $e$ also forms arcs $(ad)$ and $(bc)$. Since the points $a,b,c,d$ for each edge $e$ are connected along the vertex set of $G$ in an identical way for both $G^{\tau (A_s)}$ and $s$, then $c(s)=f(G^{\tau (A_s)})$. The proposition then follows.
\end{proof}

Many graph polynomials satisfy natural duality relations (for example, the Tutte polynomial satisfies $T(G;x,y)=T(G^*;y,x)$ for plane graphs).  In the following result, we give a duality, or rather a twisted duality, identity for the Penrose polynomial.
\begin{theorem}\label{drel}
Let $G$ be an embedded graph and $e$ be {\em any} edge of $G$. Then the Penrose polynomial satisfies the following twisted duality relation: 
 \[ P(G;\la) = P(G^{\delta (e)};\la) -P(G^{\delta \tau (e)};\la) = P(G^{\delta (e)};\la) + P(G^{\delta \tau \delta(e)};\la) .\]
\end{theorem}
\begin{proof}
We begin with the equation 
\begin{equation}\label{eq.peng2}
q(G_m; W_P, \la) = q(G_m;W_P^{\delta(e)}, \la)-q(G_m;W_P^{\tau\delta(e)}, \la).
\end{equation}
To see why this equation holds, note that by applying Proposition~\ref{p.recQ} to the weight systems $W_P$, $W_p^{\delta(e)}$, and $W_P^{\tau\delta(e)}$,  respectively, we have
\[ q(G_m; W_P, \lambda) = q((G_m)_{{wh}(v_e)} ; W_P, \lambda)-q((G_m)_{{cr}(v_e)} ; W_P, \lambda),
\]
\[ q(G_m; W_P^{\delta(e)}, \lambda) = q((G_m)_{{bl}(v_e)} ; W_P, \lambda)-q((G_m)_{{cr}(v_e)} ; W_P, \lambda),
\]
and 
\[ q(G_m; W_P^{\tau\delta(e)}, \lambda) = - q((G_m)_{{wh}(v_e)} ; W_P, \lambda)+q((G_m)_{{bl}(v_e)} ; W_P, \lambda).
\]
Substituting the three identities above into the left and right hand sides of Equation~\eqref{eq.peng2} will give the required equality.

We will now express each of the terms in Equation~\eqref{eq.peng2} in terms of the Penrose polynomial.
By Proposition~\ref{p.petr}, we have
\begin{equation}\label{eq.peng12}
P(G;\la)=q(G_m; W_P, \la).
\end{equation}
For the second term we have
\begin{equation}\label{eq.peng13}
\begin{split}
 P(G^{\delta (e)}; \la) 
 &= q((G^{\delta(e)})_m; W_P, \la) \\ &=
 q(((G^{\delta (e)})^{\delta (e)})_m; W_P^{\delta(e)}, \la) \\
& =  q(G_m ; W_P^{\delta(e)},\la),
\end{split}
\end{equation}
where the first equality follows by Proposition~\ref{p.petr}, the second follows from Theorem~\ref{t.qsd}, and the third follows from the fact that $(\delta(e))(\delta (e)) = 1(e)$.

Finally, we can rewrite third term of Equation~\eqref{eq.peng2}:
\begin{equation}\label{eq.peng14}
\begin{split}
 P(G^{\delta\tau (e)}; \la) 
 &= q((G^{\delta\tau(e)})_m; W_P, \la) \\ &=
 q(((G^{\delta\tau (e)})^{\tau\delta (e)})_m; W_P^{\tau\delta(e)}, \la) \\
& =  q(G_m ; W_P^{\tau\delta(e)},\la).
\end{split}
\end{equation}
Here, the first equality follows by Proposition~\ref{p.petr}, the second follows from Theorem~\ref{t.qsd}, and the third from the fact that $(\delta\tau(e))(\tau\delta (e)) = 1(e)$.

The result stated in the proposition then follows by substituting the identities in Equations~\eqref{eq.peng12}, \eqref{eq.peng13} and ~\eqref{eq.peng14} into Equation~\eqref{eq.peng2}. 

The final form in the statement of the theorem is obtained by applying the first to $G^{\delta (e)}$.

\end{proof}

\begin{theorem}\label{p.penid}
 The Penrose polynomial of an embedded graph $G$ has the following properties.
\begin{enumerate}
\item \label{651} If $A\subseteq E(G)$, then $ P(G;\lambda) = (-1)^{|A|}P(G^{\tau(A)};\lambda)$, and in particular $|P(G;\lambda)|$ is an invariant of the orbits of the half-twist, (\emph{i.e.} of partial Petrials). Furthermore, $ P(G;\lambda) = (-1)^{|A|}Q(G; (\boldsymbol1, \boldsymbol0, \boldsymbol{-1})^{\tau(A)}, \lambda)$. 
\item If $e$ is a non-twisted  loop of $G$ that bounds a 2-cell, then $P(G;\la) = (\la-1)P(G-e;\la)$. 
\item \label{4term fig} The Penrose polynomial satisfies the {\em four-term relation}: 
\[     
P\left(\raisebox{-6mm}{\includegraphics[height=14mm]{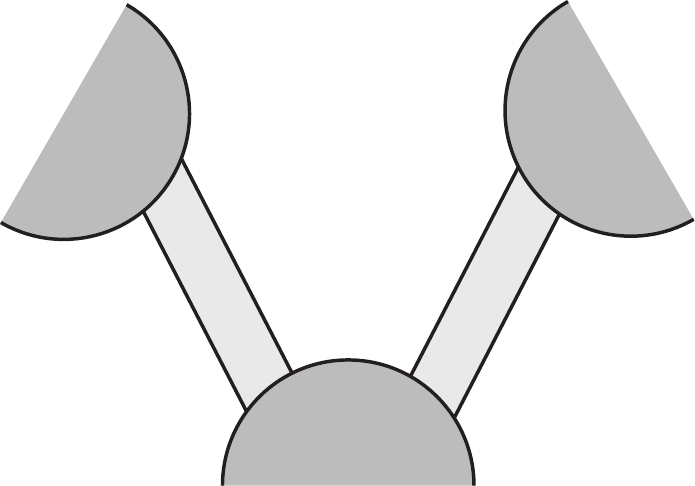}} \; ;  \la\right) -
P\left(\raisebox{-6mm}{\includegraphics[height=14mm]{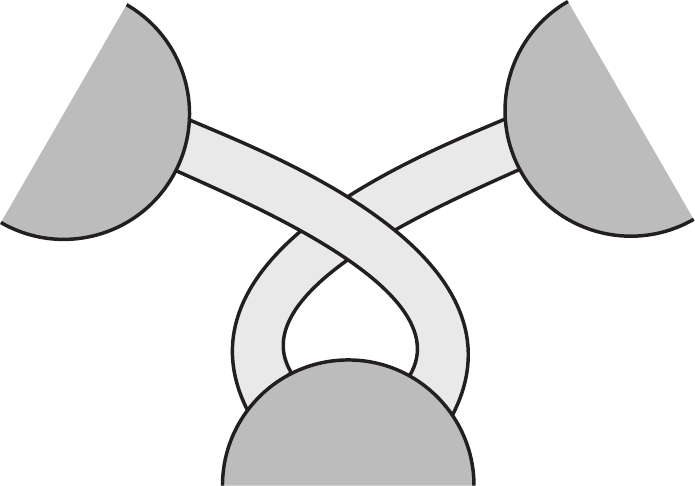}} \; ;  \la\right) =
P\left(\raisebox{-6mm}{\includegraphics[height=14mm]{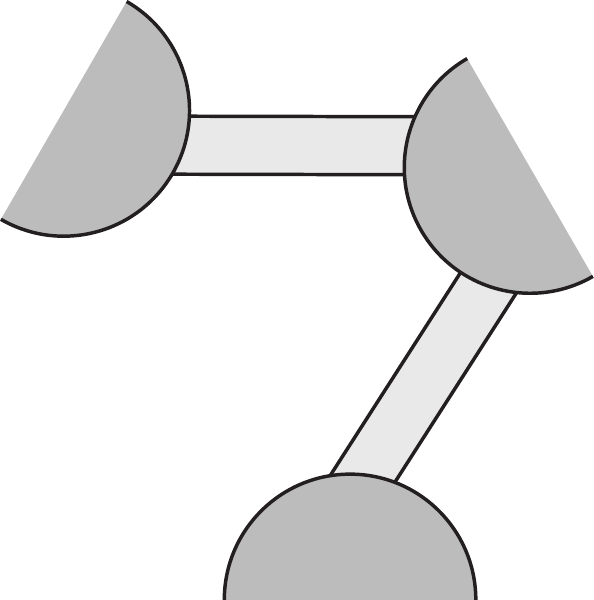}} \; ;  \la\right) -
P\left(\raisebox{-6mm}{\includegraphics[height=14mm]{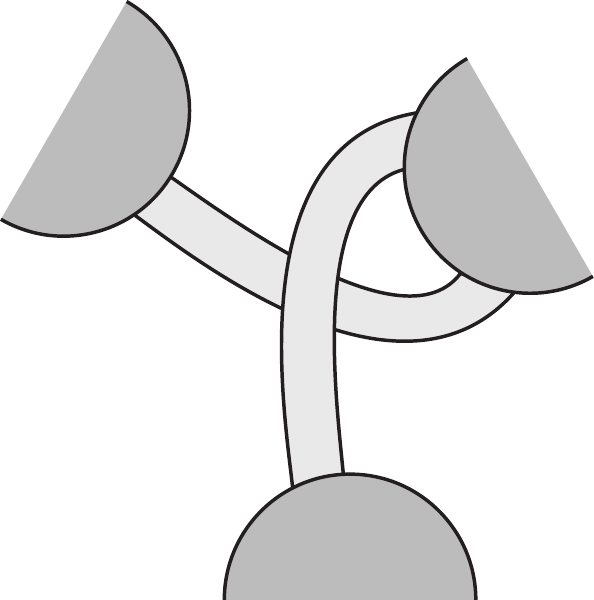}} \; ;  \la\right), 
\]
where the four ribbon graphs in the figure are identical except in the region shown.

\end{enumerate}
\end{theorem}
\begin{proof} We will prove the properties one by one.
\begin{enumerate}
\item We have that
\begin{multline*} P(G^{\tau(A)};\la)=  Q(G^{\tau(A)}; (\boldsymbol1, \boldsymbol0, \boldsymbol{-1}), \la) =  Q((G^{\tau(A)})^{\tau(A)}; (\boldsymbol1, \boldsymbol0, \boldsymbol{-1})^{\tau(A)}, \la) =  Q(G; (\boldsymbol1, \boldsymbol0, \boldsymbol{-1})^{\tau(A)}, \la), \end{multline*} 
where the first equality follows from Proposition~\ref{p.petr}, the second from Theorem~\ref{t.qsd}, and the third follows from the fact that $\tau(A)\tau(A)=1(A)$.
It remains to show that 
\[  P(G^{\tau(A)};\la)=  (-1)^{|A|} Q(G; (\boldsymbol1, \boldsymbol0, \boldsymbol{-1}), \la). \]
This identity follows by recalling the skein definition of the Penrose weight system:
\[ \raisebox{0mm}{$W^{\tau(A)}_P(G_m): $} \quad \quad
\raisebox{-4mm}{ \includegraphics[height=10mm]{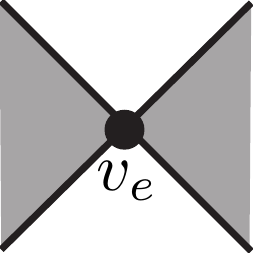}} \;\; = 
\left\{ \begin{array}{lrlr}
   \raisebox{-5mm}{ \includegraphics[height=10mm]{m2}}  &-&  \raisebox{-5mm}{ \includegraphics[height=10mm]{m4}} & \text{ if } e\notin A \\ && \\
\raisebox{-5mm}{ \includegraphics[height=10mm]{m4}}  &- & \raisebox{-5mm}{ \includegraphics[height=10mm]{m2}} &  \text{ if } e\in A  
\end{array}
\right. ,
 \]
and observing that the relations differ only by a factor of $-1$.  This proves both parts of Item (\ref{651}).

\item The proof of this property is routine and therefore omitted.

\item   
Let $G_i$, for $i=1,\ldots , 4$, be the four embedded graphs shown  in the four-term relation of Item~\ref{4term fig} in the order shown in the defining figure. There is a natural bijection between their edge sets. Identify the edge sets  of each of the  $G_i$'s using this correspondence,
letting $e$ and $f$ denote the distinguished edges of the $G_i$'s shown in the figure. Since by Proposition \ref{l.peng}
\[ P(G_i;\la) =   \sum_{\substack{A\subseteq E(G_i)-\{e,f\} \\ B\subseteq \{e,f\}}} (-1)^{|A\cup B|} \lambda^{f(G_i^{\tau(A\cup B)})}  ,
  \]
it is enough to show that for a fixed subset $A$ of   $E(G_i)-\{e,f\}$, we have 
\begin{multline*} \sum_{B\subseteq \{e,f\} }  (-1)^{|B|} \lambda^{f(G_1^{\tau(A\cup B)})}  
- \sum_{B\subseteq \{e,f\} }  (-1)^{|B|} \lambda^{f(G_2^{\tau(A\cup B)})} 
- \sum_{B\subseteq \{e,f\} }  (-1)^{|B|} \lambda^{f(G_3^{\tau(A\cup B)})} \\
+ \sum_{B\subseteq \{e,f\} }  (-1)^{|B|} \lambda^{f(G_4^{\tau(A\cup B)})} =0.
 \end{multline*}
This identity is easily verified by calculation. 
\end{enumerate}
\end{proof}

The fact that the state sum defining the Penrose polynomial satisfies the four-term relation is well known in the theory of Vassiliev invariants of knots (see \emph{e.g.} \cite{BN95}) that we are now able to recognize in the context of the topological Penrose polynomial. This result appears in knot theory  since the Penrose polynomial can be expressed in terms of the $\mathfrak{so}_N$ weight system.

\begin{corollary}\label{c.sp}
If $G$ is a self-Petrial embedded graph, then $ P(G;\lambda)=0 $ or   $|E(G)|$ is even. 
\end{corollary}
\begin{proof}
$G$ is  self-Petrial  if and only if $G^{\tau(E(G))}=G$.
By Item~\eqref{651} of Proposition~\ref{p.penid}, we then have  $ P(G;\lambda) = (-1)^{|E(G)|}P(G^{\tau(E(G))};\lambda)$, and the result follows.
\end{proof}

\begin{example}\label{hemi}
In Section~\ref{distinct} it was noted that a plane graph has trivial Penrose polynomial if and only if it contains a bridge.  While $P(G;\lambda)=0$ for all embedded graphs that contain a bridge, the converse does not hold in general. That is, a trivial Penrose polynomial does not ensure that an embedded graph contains a bridge. For example, the hemidodecahedron is    bridgeless, but,  as it has an odd number of edges and   is self-Petrial,  it follows from Corollary~\ref{c.sp}  that its Penrose polynomial is zero. 
\end{example}

\subsection{A deletion-contraction relation for the Penrose polynomial}

We emphasize that in the following theorem the contraction of a loop $e$, or any other edge $e$ for that matter, is defined by $G/e:=G^{\delta (e)}-e$ as described in Subsection~\ref{ss.delcon}.

By moving into the broader class of embedded graphs, we now see that the Penrose polynomial actually does have  deletion-contraction reductions, like those for the Tutte polynomial, albeit with a slight twist.

\begin{theorem}\label{t.delcon} Suppose $G$ is an embedded graph, and $e \in E(G)$. Then: 
\begin{enumerate}
\item  \label{661} 
\[ P(G;\la) = P(G/e; \la) - P(G^{\tau(e)}/e; \la);\]
\item  \label{662}
\[ P(G;\la) = P(G/e ;\la) - P(G^{\tau\delta (e)}/e; \la); \]
\item  \label{663}
\[
P(G^{\delta\tau(e)};\la)= P(G-e; \la)-P(G/e; \la).
\]
\end{enumerate}
\end{theorem}

\begin{proof} Items (\ref{661}) and (\ref{662}) follow immediately from Theorem~\ref{t.transdl} and Corollary~\ref{c.trans delete contract} and the definition of contraction. (Item (\ref{662}) also follows from (\ref{661}) since  $G^{\tau \delta ( e)}/e =   G^{\delta \tau \delta ( e)}-e =  G^{ \tau \delta \tau ( e)}-e = G^{ \delta \tau ( e)}-e = G^{ \tau ( e)}/e.$) For Item (\ref{663}), 
observe that $ G/e= G^{\delta (e)}-e =G^{\tau\delta (e)} -e$, and so \eqref{662} gives
$P(G;\la)= P(G^{\tau\delta(e)}-e; \la)-P(G^{\tau\delta(e)}/e; \la)$. Item~\eqref{663} then follows  by applying this identity to $G^{\delta\tau(e)}$.

\end{proof}

\section{Connections with the circuit partition polynomial and some evaluations}\label{s.cpp}

An Eulerian digraph is an Eulerian graph with the edges directed so that the in-degree and out-degree are equal at each vertex.  The circuit partition polynomial of an Eulerian digraph, $j(\vec{G};x)$, defined in \cite{E-M98}, may be given by $j(\vec{G};x) =q(G;W_D, x)$, where $G$ is the underlying undirected graph of $\vec{G}$ and $W_D$ is the weight system that assigns a $1$ to a pair of edges at a vertex if one is directed into and the other out of the corresponding vertex in $\vec{G}$ .  Pictorially, this is:

\[\text{if }    \raisebox{-2mm}{$\begin{array}{c} \includegraphics[height=10mm]{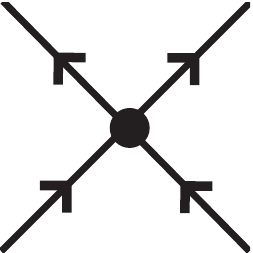}  \\ v\in \vec{G}\end{array}$},
\text{ then} \raisebox{-2mm}{$\begin{array}{c} \includegraphics[height=10mm]{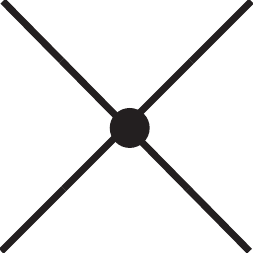}  \\ v\in G\end{array}$}\;\; \longrightarrow   \;\;1 \cdot   \raisebox{-4mm}{\includegraphics[height=10mm]{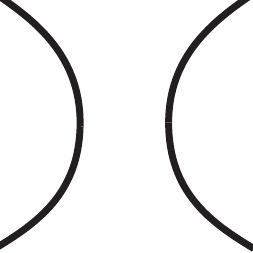}} + 0\cdot      \raisebox{-4mm}{\includegraphics[height=10mm]{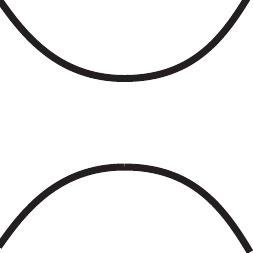}} +1\cdot     \raisebox{-4mm}{ \includegraphics[height=10mm]{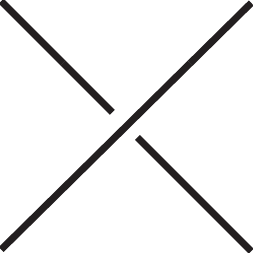}}. \]

A number of combinatorial interpretations are known for evaluations of $j(\vec{G};x)$.  We  will now relate the Penrose polynomial and the circuit partition polynomial in the special case of orientable checkerboard colourable (\emph{i.e.} properly face $2$-colourable) graphs.  This gives new interpretations for the Penrose polynomial for this class of graphs, which in turn motivates interpretations for broader classes of graphs.

The connection between the circuit partition polynomial and the Penrose polynomial for an oriented checkerboard coloured graph $G$ depends on a particular Eulerian digraph based on the medial graph, which we will refer to as the \emph{Penrose directed medial graph} and denote by $\pd$.  The digraph $\pd$ is the graph $G_m$ with the edges directed clockwise as they follow the white faces of $G$ and counter clockwise as they follow the black faces of $G$  (these are the black and white faces of the checkerboard colouring of $G$, not faces in the canonical checkerboard colouring of $G_m$).  Note that the Penrose states of $G_m$ correspond exactly to the non-zero weighted states of $G_m$ in the weight system $W_D$ induced by $\pd$.

We begin with the following lemma, which is an adaptation of Proposition~3 of \cite{Ai97}.

\begin{lemma}\label{l.cpp}
Let $G$ be an orientable checkerboard colourable graph. Then, if $s\in \mathcal{P}(G_m)$ is a Penrose state of $G_m$, we have $ cr(s)+c(s)\equiv f(G) \mod 2 $. 

\end{lemma}

\begin{proof}
We begin by reformulating the result in terms of twisted duals. Let $s\in \mathcal{P}(G_m)$,   and let  $B_s$ denote the set of  edges of $G$ at which there is a non-crossing vertex state at the corresponding vertex in the state $s$. We immediately have that $cr(s)=e(G)-|B_s|$.
In addition, since $G^{\tau\delta (E(G))}$ and $G^{\tau (E(G))}$ have the same boundary components, $c(s)=f(G^{\tau\delta (E(G))} -B_s)$.

Since $G$ is orientable and $G^*$ is bipartite, it follows from Proposition~4.30 of \cite{E-MMd} that $G^{\tau\delta (E(G))}$ is orientable. Thus Euler's formula gives 
\begin{equation}\label{e.pc} 
v(G^{\tau\delta (E(G))}-B_s) - e(G^{\tau\delta (E(G))}-B_s)+f(G^{\tau\delta (E(G))}-B_s) = 2k(G^{\tau\delta (E(G))}-B_s)+2g(G^{\tau\delta (E(G))}-B_s),
\end{equation} 
where genus  $g(G^{\tau\delta (E(G))}-B_s)$  is obtained by summing the genus of the connected components of $G^{\tau\delta (E(G))}-B_s$. 

Making the substitutions $  v(G^{\tau\delta (E(G))}-B_s) =f(G)$, $e(G^{\tau\delta (E(G))}-B_s)=e(G)-|B_s|=cr(s)$, and $f(G^{\tau\delta (E(G))} -B_s)=c(s)$ in Equation~\eqref{e.pc}, gives 
\[  f(G)-cr(s)+c(s) =  2k(G^{\tau\delta (E(G))}-B_s)+2g(G^{\tau\delta (E(G))}-B_s).\] Considering  this equation modulo $2$ gives the result.
\end{proof}

\begin{theorem}\label{t.cpp}
Let $G$ be an oriented  checkerboard coloured graph.  Then  
\[ P(G, \la)=(-1)^{f(G)}  j(\pd, -\la).   \]
\end{theorem}

\begin{proof}
With the given orientation of $\pd$, the states of $\pd$ are precisely the Penrose states of $G_m$. Thus 
\begin{multline*}P(G;\la)  =
         \sum_{s \in \mathcal{P}(G_m )}
           \left(  - 1 \right)^{cr\left( s\right)}
            \la^{c(s)}   = 
             \sum_{s \in \mathcal{S}(\vec{G}_m )}
            \left( { - 1} \right)^{cr\left( s\right)+c(s)}
           (-\la)^{c(s)}  
            \\
            = (-1)^{f(G)}              \sum_{s \in \mathcal{S}(\vec{G}_m )}
                       (-\la)^{c(s)}  
                      =(-1)^{f(G)} j(\vec{G}_m; -\la),
                 \end{multline*}
where the second last equality follows by Lemma~\ref{l.cpp}.
\end{proof}

We can now extend some results listed in Subsection~\ref{distinct} to more general classes of graphs.

\begin{proposition} If $G$ is orientable and checkerboard colourable, then $P(G;2) = 2^{v(G)}$.
\end{proposition}

\begin{proof}
By theorem \ref{t.cpp}, $P(G;2)=(-1)^{f(G)}j(\pd;-2)$.  However, in \cite{E-M04} it was shown that if $\vec{G}$ is a 4-regular Eulerian digraph, then $j(\vec{G};-2)=(-1)^{v(\vec{G})}(-2)^{h(\vec{G})}$, where $h(\vec{G})$ is the number of components in the anticircuit state of $\vec{G}$, that is, the state that results from pairing the two incoming edges and pairing the two outgoing edges at each vertex, so the arrows alternate along the edges of each circuit.

In a medial graph $G_m$ with a Penrose orientation, the anticircuits simply follow the black faces in the canonical checkerboard colouring of $G_m$, and thus $h(\pd)=v(G)$.  Furthermore $v(\pd)=e(G)$.

Thus, $P(G;2)=(-1)^{f(G)+e(G)}(-2)^{v(G)}=(-1)^{f(G)-e(G)+v(G)}2^{v(G)}=(-1)^{2k(G)-\gamma(G)}2^{v(G)}$.  Since $G$ is orientable, $\gamma(G)$ is even, and the result follows.
\end{proof}

\begin{proposition} If $G$ is orientable and checkerboard colourable, then $P(G;-1) = (-1)^{f(G)}2^{e(G)}$.
\end{proposition}

\begin{proof}
By Theorem \ref{t.cpp}, $P(G;-1)=(-1)^{f(G)}j(\pd;1)$.  However, from \cite{E-M04},  if $\vec{G}$ is a 4-regular Eulerian digraph, then $j(\vec{G};1)=\prod_v{(\frac{deg(v)}{2})!}$. Since $\pd$ is 4-regular, the result follows.
\end{proof}

As noted in Subsection~\ref{distinct}, these two results do not extend to general embedded graphs.

We can now use the connection in Theorem~\ref{t.cpp}, and give combinatorial interpretations for the Penrose polynomial at all negative integers for orientable checkerboard colourable graphs, and to motivate interpretations for all positive integers for all embedded graphs.   These interpretations are in terms of special edge colourings of the medial graph extended from those described by Jaeger in \cite{Ja90} to give combinatorial interpretations for the Penrose polynomial of a plane graph at positive integers.   We will recover Jaeger's result as a corollary.

We begin with the special edge colouring, called a $k$-valuation, extended to general embedded graphs.
\begin{definition}
Let $G$ be an embedded graph and $G_m$ be its canonically checkerboard coloured medial ribbon graph. A {\em $k$-valuation} of $G_m$ is a $k$-edge colouring $\phi: E(G_m)\rightarrow \{1, 2, \ldots , k\}$ such that for each $i$ and every vertex $v_e$ of $G_m$, the number of $i$-coloured edges incident with $v_e$ is even.

A $k$-valuation is said to be {\em admissible}  if, at each vertex of $G_m$, the $k$-valuation is of one of the following two types:
\[\includegraphics[width=1.3cm]{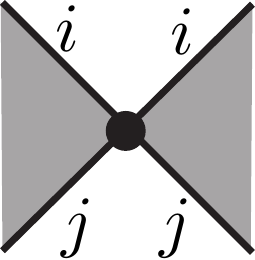}\hspace{1cm}\raisebox{5mm}{\text{or}} \hspace{1cm} \includegraphics[width=1.3cm]{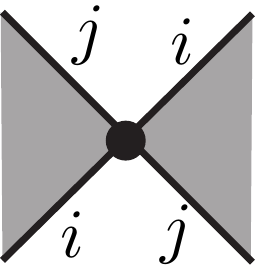},\]
where $i\neq j$.  The configuration on the left corresponds to a white split state, and the one on the right to a crossing state.

A $k$-valuation is said to be {\em permissible} if at each vertex of $G_m$, the $k$-valuation is of one of the two types shown above, or of the following type, called {\em total}:
\[\includegraphics[width=1.3cm]{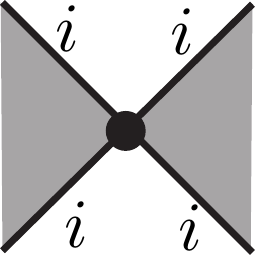}.\]
\end{definition}

We begin with a new combinatorial interpretation for the Penrose polynomial evaluated at negative integers.  A different, recursive, formula for plane graphs may be found in \cite{E-MS01, Sar01}.

\begin{proposition}\label{total}  If $G$ is orientable and checkerboard colourable, then $P(G;-n) = (-1)^{f(G)}\sum2^{m(s)}$, where the sum is over all permissible $n$-valuations $s$ of $G_m$ and where $m(s)$ is the number of total vertices in $s$.
\end{proposition}

\begin{proof}
By Theorem \ref{t.cpp}, $P(G;-n)=(-1)^{f(G)}j(\pd;n)$.  However, from \cite{E-M04},  if $\vec{G}$ is a 4-regular Eulerian digraph, then $j(\vec{G};n)=\sum2^{m(c)}$, where $c$ is an Eulerian $n$-colouring, that is, an edge colouring of $\vec{G}$ such that the restriction to any one colour is an Eulerian digraph (not necessarily connected), and where $m(c)$ is the number of total vertices.  Since this exactly corresponds to a permissible $n$-colouring, the result follows.
\end{proof}

Proposition \ref{total} does not hold for general embedded graphs.  For example, any graph with a bridge has a Penrose polynomial of zero.

The following theorem, which expresses the Penrose polynomial of a plane graph in terms of $k$-valuations, is due to Jaeger (see \cite{Ja90} Proposition~13, and also see  \cite{Ai97} Proposition~4).

\begin{theorem}[Jaeger \cite{Ja90}]\label{th.addval} 
If $G$ is a plane graph, then for each  $k\in \mathbb{N}$, $P(G;k)$ is equal to the number of admissible $k$-valuations of  $G_m$.
\end{theorem}

This evaluation of the Penrose polynomial does not extend to general embedded graphs. For example, the embedded graph in the torus given by
\raisebox{-3mm}{\includegraphics[height=8mm]{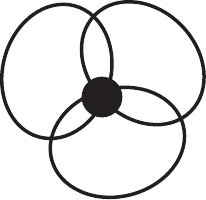}}
has Penrose polynomial $- \lambda^3+4 \lambda^2-3 \lambda $, but its medial graph has $k^3-2k^2+k$ admissible $k$-valuations.  However, we are able to give a more general formulation, motivated by an identity for the circuit partition polynomial at negative integers in \cite{E-M04}, one that does hold for all embedded graphs, and that reduces to Theorem \ref{th.addval} for plane graphs.

\begin{theorem}\label{sumcross}  If $G$ is an embedded graph, then $P(G;n) = \sum{(-1)^{cr(s)}}$, where the sum is over all admissible $n$-valuations $s$ of $G_m$ and where $cr(s)$ is the number of crossing states in $s$.
\end{theorem}

\begin{proof}
We induct on the edges of $G$.  The result clearly holds if $G$ has no edges, so let $e$ be an edge of $G$.  We will use the following observations from \cite{E-MMe}.
Let $G$ be an embedded graph with embedded medial graph $G_m$, and $e$ be an edge of $G$. Then
\begin{enumerate}
\item  $(G_m)_{bl(v_e)} = (G-e)_m$;
\item $(G_m)_{wh(v_e)}=(G/e)_m$;
\item $(G_m)_{cr(v_e)}$  and  $(G^{\tau(e)}/e)_m$ are twists of each other.
\end{enumerate}

We note that the admissible colourings of a graph are unaffected by twists, since the cyclic order of edges about each vertex is preserved.  Thus, the admissible colourings of $(G_m)_{cr(v_e)}$  and  $(G^{\tau(e)}/e)_m$ are the same.

By Theorem~\ref{t.delcon}, 
\[ P(G;\la) = P(G/e; \la) - P(G^{\tau(e)}/e; \la).\]
By induction, the right-hand side is
\[\sum_{(G/e)_m}{(-1)^{cr(s)}}-\sum_{(G^{\tau(e)}/e)_m}{(-1)^{cr(s)}},\] 
where the first sum is over all admissible $n$-valuations $s$ of $(G/e)_m$ and the second sum is over all admissible $n$-valuations $s$ of $(G^{\tau(e)}/e)_m$.

As noted above, this becomes
\[\sum_{(G_m)_{wh(v_e)}}{(-1)^{cr(s)}}-\sum_{(G_m)_{cr(v_e)}}{(-1)^{cr(s)}},\] 
where the first sum is over all admissible $n$-valuations $s$ of $(G_m)_{wh(v_e)}$ and the second sum is over all admissible $n$-valuations $s$ of $(G_m)_{cr(v_e)}$.

In an  admissible $n$-valuation of $(G_m)_{wh(v_e)}$, the two arcs forming the white transition can  either be assigned different colours, or both be assigned the same colour,
 and similarly in the crossing case.  We can separate the sums accordingly:
\[\mathop{\sum_{(G_m)_{wh(v_e)} }}_{\text{diff}}(-1)^{cr(s)}
+  \mathop{\sum_{(G_m)_{wh(v_e)} }}_{\text{same}}(-1)^{cr(s)}
-\mathop{\sum_{(G_m)_{cr(v_e)} }}_{\text{diff}}{(-1)^{cr(s)}}-\mathop{\sum_{(G_m)_{cr(v_e)} }}_{\text{same}}{(-1)^{cr(s)}}.\] 

The first sum corresponds to a sum over all admissible $n$ valuations of $G_m$ with a split state at $v_e$.  The second and fourth sums both correspond to all permissible $n$ valuations of $G_m$ with exactly one total vertex at $v_e$, and hence cancel.  The third sum corresponds to a sum over all admissible $n$ valuations of $G_m$ with a crossing state at $v_e$, but with one more crossing to be counted, thus changing the subtraction to addition. 

This yields that $P(G;n) = \sum{(-1)^{cr(s)}}$, as desired.

\end{proof}

Theorem \ref{th.addval} follows as an immediate corollary, using the Jordan curve theorem.  If $G$ is plane, then in any admissible colouring,  $G_m|_{E_i}$, the subgraph induced by the edges of colour  $i$, is a disjoint union of simple closed curves in the plane.  By the Jordan curve theorem, the intersection of $G_m|_{E_i}$ and $G_m|_{E_j}$ for any $i \neq j$ must be an even number of points.  Thus, $cr(s) \equiv 0$ mod 2 for every admissible colouring $s$ of $G_m$.

\section{Colourings and the Penrose polynomial}\label{s.colour}
We now use the ribbon group action to develop connections between the Penrose polynomial and proper graph colourings. In particular, we will show how the  Penrose polynomial of a plane graph can be written as a sum of the chromatic polynomials  of its twisted duals. This results completes a theorem of Aigner where the Penrose polynomial was shown to give an upper bound on the number of proper graph colourings. This extension of Aigner's result can not be realized by the original Penrose polynomial which is restricted to plane graphs.

\begin{theorem}[Aigner \cite{Ai97}]  \label{t.aigner}
Let $G$ be a plane graph, then for all $k\in \mathbb{N}$ we have  \[\chi(G^*;k)\leq P(G;k).\]
\end{theorem}

In Theorem~\ref{t.pac}, we will  complete Theorem~\ref{t.aigner} by showing that the Penrose polynomial of a plane graph $G$ is in fact {\em equal} to a sum of specific chromatic polynomials, that is, $P(G;\lambda) = \sum_{A\subseteq E(G)}  \chi ((   G^{\tau(A)}   )^*   ;\lambda)  $.   Thus, the expression $\chi(G^*;k)$ is just one summand in the full expression for the Penrose polynomial $P(G;k)$ given in Theorem~\ref{t.pac}.  Theorem~\ref{t.aigner} then follows from Theorem~\ref{t.pac} as a corollary.

\begin{theorem}\label{t.pac}
If $G=(V(G),E(G))$ is a plane graph, then 
\[  P(G;\lambda) = \sum_{A\subseteq E(G)}  \chi ((   G^{\tau(A)}   )^*   ;\lambda)  .\] 
\end{theorem}

\begin{proof}[Proof of Theorem~\ref{t.pac}]

If $G$ is a plane graph, then, by Theorem~\ref{th.addval}, $P(G;k)$ is equal to the number of admissible $k$-valuations for all $k\in \mathbb{N}$:
\begin{equation}\label{pc1}P(G;k) =  ( \text{number of admissible $k$-valuations}).  \end{equation}
 We construct a bijection from the set of admissible $k$-valuations of the medial graph $G_m$ to a certain set of colourings of the boundaries of the partial Petrials  of $G$. To do this we view $G$ as a ribbon graph.  Given an admissible $k$-valuation, $\phi$, of $G_m$,  let $\{v_e\,|\, e\in A_{\phi}\}$  be the set of vertices of $G_m$ at which a crossing state is assigned in the $k$-valuation. Note that the indexing set, $A_{\phi}$, is a set of edges of $G$. Observe that the cycles in $G_m$ (which are determined by the colours  in the $k$-valuation $\phi$) follow exactly the boundary components of the partial Petrial $G^{\tau(A_{\phi})}$. Moreover, the colours of the  cycles in the $k$-valuation induce a colouring of the boundary components of $G^{\tau(A_{\phi})}$.   We define a {\em proper boundary $k$-colouring } of a ribbon graph to be a map from its set of boundary components to the colours $ \{1,2,\ldots , k\}$ with the property that whenever  two boundary components share a common edge, they are assigned different colours. It is then clear that the map from $\phi$ to $G^{\tau(A_{\phi})}$ defines a bijection between the set of admissible $k$-valuations of $G_m$ and the set of proper boundary $k$-colourings of the partial Petrials of $G$:
\begin{equation}\label{pc2}
 (\text{no. of admissible $k$-valuations of } G_m)=  \sum_{A\subseteq E(G)}  ( \text{no. of proper boundary $k$-colourings of } G^{\tau(A)} )
\end{equation}
Finally, we observe that a proper boundary $k$-colouring of $G^{\tau(A)}$  corresponds to a proper face $k$-colouring of $G^{\tau(A)}$ when $G^{\tau(A)}$ is viewed as cellularly embedded, and hence to a proper $k$-colouring of $(G^{\tau(A)})^*$.  Thus, the number of boundary $k$-valuations of a ribbon graph $G^{\tau(A)}$ is equal to $\chi((G^{\tau(A)})^*;k)$, for $k\in \mathbb{N}$:
\begin{equation}\label{pc3}
  \sum_{A\subseteq E(G)}  ( \text{number of boundary $k$-valuations of  } G^{\tau(A)} )  = \sum_{A\subseteq E(G)}   \chi ((   G^{\tau(A)}   )^*   ;k). 
\end{equation}
The result then follows from Equations~\eqref{pc1}, \eqref{pc2},and  \eqref{pc3}.
\end{proof}

We observe that Theorem~\ref{t.pac} does not hold for  arbitrary non-plane graphs (for example, consider the graph embedded on the Klein bottle that consists of one vertex and two edges $e$ and $f$ with the cyclic order $efef$ at the vertex). Nor does Theorem~\ref{t.pac} provide a characterization of plane graphs (for example, consider the graph embedded on the torus that consists of one vertex and two edges $e$ and $f$ with the cyclic order $efef$ at the vertex).

We conclude by using Theorem \ref{t.pac} to reformulate the Four Colour Theorem. 
\begin{corollary}
The following statements are equivalent:
\begin{enumerate}
\item the Four Colour Theorem is true;
\item for every connected, bridgeless plane graph $G$ there exists $A\subseteq E(G)$ such that $\chi((G^{\tau(A)})^*;3)\neq0$;
\item for every connected, bridgeless plane graph $G$ there exists $A\subseteq E(G)$ such that $\chi((G^{\tau(A)})^*;4)\neq0$.
\end{enumerate}
\end{corollary}
\begin{proof}
Corollary~9 of \cite{Ai97} states  that the Four Colour Theorem is equivalent to showing that $P(G;3)>0$ or $P(G;4)>0$ for all connected, bridgeless plane graphs $G$. Since $\chi(H;k)\geq 0$ for all $k\in \mathbb{N}$ and graphs $H$, Theorem~\ref{t.pac} tells us that  $P(G;k)>0$ if and only if one of the summands $\chi((G^{\tau(A)})^*;k)\neq0$. The result then follows.
\end{proof}

\bibliographystyle{amsplain} 

\end{document}